\documentclass{amsart}
\usepackage{amsthm, amssymb, amsfonts, amscd}
\usepackage{graphicx}
\usepackage{tikz-cd}
\usepackage{verbatim}
\usepackage{enumitem}
\usepackage{mathrsfs, mathtools}
\usepackage{xcolor}
\usepackage[colorlinks = true,
            linkcolor = black,
            urlcolor  = blue,
            citecolor = black]{hyperref}
\usepackage{cleveref}
\usepackage[margin=1in]{geometry}
\usepackage{setspace}
\usepackage{diagbox}
\allowdisplaybreaks

\theoremstyle{definition} 
\newtheorem{thm}{Theorem}[section]
\newtheorem{cor}[thm]{Corollary}
\newtheorem{prop}[thm]{Proposition}
\newtheorem{lem}[thm]{Lemma}

\newtheorem{rmk}[thm]{Remark}

\newcommand{\mr}{\mathrm}
\newcommand{\mf}{\mathfrak}
\newcommand{\mb}{\mathbf}
\newcommand{\mc}{\mathcal}

\newcommand{\Aut}{\mr{Aut}}

\newcommand{\bb}{\mathbb}

\newcommand{\bZ}{\bb{Z}}
\newcommand{\bQ}{\bb{Q}}
\newcommand{\bR}{\bb{R}}
\newcommand{\bC}{\bb{C}}
\newcommand{\bP}{\bb{P}}

\newcommand{\bF}{\bb{F}}
\newcommand{\bE}{\bb{E}}
\newcommand{\fp}{\bF_p}

\newcommand{\zp}{\bZ_p}

\newcommand{\ze}{\zeta}
\newcommand{\al}{\alpha}
\newcommand{\be}{\beta}
\newcommand{\ga}{\gamma}

\newcommand{\ol}{\overline}
\newcommand{\Hom}{\mr{Hom}}
\newcommand{\M}{\mr{M}}
\newcommand{\cok}{\mr{cok}}
\newcommand{\Sur}{\mr{Sur}}
\newcommand{\rank}{\mr{rank}}

\newcommand{\lt}{\left |}
\newcommand{\rt}{\right |}
\newcommand{\lf}{\left \lfloor}
\newcommand{\rf}{\right \rfloor}
\begin{document}

\newcommand{\lcom}[1]{{\color{red}{Jungin: #1} }}
\newcommand{\jcom}[1]{{\color{blue}{Jiwan: #1} }}
\newcommand{\ycom}[1]{{\color{orange}{Myungjun: #1} }}

\title[Sharp threshold for classical random $p$-adic matrices]{Sharp threshold for universality of cokernels of classical random matrix models over the $p$-adic integers}
\author{Jiwan Jung, Jungin Lee and Myungjun Yu}
\date{}
\address{J. Jung -- Department of Mathematics, Pohang University of Science and Technology, Pohang 37673, Republic of Korea \newline
J. Lee -- Department of Mathematics, Ajou University, Suwon 16499, Republic of Korea \newline M. Yu -- Department of Mathematics, Yonsei University, Seoul 03722, Republic of Korea}
\email{guinipig123@postech.ac.kr, jileemath@ajou.ac.kr, mjyu@yonsei.ac.kr}

\begin{abstract}
We prove that $\frac{\log n}{n}$ is the sharp threshold for universality of the distribution of cokernels of random matrices over $\mathbb{Z}_p$. More precisely, let $\alpha_n = \frac{c\log n}{n}$ for a constant $c>0$ and let $A(n)$ be an $\alpha_n$-balanced random matrix over $\mathbb{Z}_p$. 
For non-symmetric, symmetric, and alternating matrix models, we prove that if $c>1$, then the limiting distribution of the cokernel of $A(n)$ coincides with the universal distribution of the corresponding symmetry type, whereas universality fails at the critical scale $c=1$.
This improves earlier universality results, which required $\alpha_n \gg \frac{\log n}{n}$, to the optimal threshold. As an application, we generalize the universality result for Sylow $p$-subgroups of sandpile groups of Erdős–Rényi random graphs to a broader class of Erdős–Rényi graph sequences.
Our approach is based on a unified framework that simultaneously treats all symmetry types of random matrices as well as the random graph model, rather than handling each case separately.
\end{abstract}
\maketitle

\section{Introduction} \label{Sec1}

For a prime $p$, let $\zp$ denote the ring of $p$-adic integers and let $\fp$ be the finite field with $p$ elements. 
For a commutative ring $R$, write $\M_n(R)$ (resp. $\mr{Sym}_n(R)$, $\mr{Alt}_n(R)$) for the set of all (resp. symmetric, alternating) $n \times n$ matrices over $R$.
Let $\al \in (0, 1/2]$ be a real number. A random element $x$ in the finite field $\fp$ is $\al$-\emph{balanced} if 
$$
\bP(x = r) \leq 1 - \al
$$ 
for every $r \in \fp$. A random element $x$ in $\zp$ or $\bZ/p^d\bZ$ is $\al$-\emph{balanced} if its reduction modulo $p$ is $\al$-\emph{balanced} as a random element in $\fp$. A random matrix in $\M_n(R)$ (where $R = \fp$, $\bZ/p^d\bZ$ or $\zp$) is $\al$-\emph{balanced} if its entries are independent and $\al$-balanced. Similarly, a random matrix in $\mr{Sym}_n(R)$ (resp. $\mr{Alt}_n(R)$)  is $\al$-\emph{balanced} if its upper triangular (resp. strictly upper triangular) entries are independent and $\al$-balanced. We note that in the definition of an $\al$-balanced matrix, the entries are not required to be identically distributed.

We first recall several universality results concerning the distribution of cokernels of random $p$-adic matrices. Under $\al$-balancedness assumptions on the entries, the distribution of the cokernel of an $n \times n$ matrix over $\zp$ converges, as $n \to \infty$, to a universal limiting distribution. In particular, the limiting distribution depends only on the symmetry type (non-symmetric, symmetric, or alternating) and not on the specific entry distributions. In the following theorems, $\cok(A(n))$ denotes the cokernel of a matrix $A(n)$ and $\Aut(H)$ denotes the automorphism group of $H$.

\begin{thm} \label{thm_nonsym}
(\cite[Theorem 4.1]{NW22}) Let $(\al_n)_{n \ge 1}$ be a sequence of positive real numbers in $(0,1/2]$ such that for every constant $\Delta > 0$, we have $\al_n \geq \frac{\Delta \log n}{n}$ for all sufficiently large $n$. Let $A(n)$ be an $\al_n$-balanced random matrix in $\M_n(\zp)$ for each $n \ge 1$. Then for every finite abelian $p$-group $H$,
\begin{equation*}
\lim_{n \to \infty} \bP(\cok(A(n)) \cong H) = \frac{1}{|\Aut(H)|} \prod_{i=1}^{\infty} (1-p^{-i}).
\end{equation*}
\end{thm}

\begin{thm} \label{thm_sym}
(\cite[Theorem 1.3]{Woo17}) Let $\al \in (0,1/2]$ and let $A(n) \in \mr{Sym}_n(\zp)$ be an $\al$-balanced random matrix for each $n \ge 1$. Then for every finite abelian $p$-group $H$,
\begin{equation*}
\lim_{n \to \infty} \bP(\cok(A(n)) \cong H) = \frac{\# \{ \text{symmetric, bilinear, perfect } \phi : H \times H \to \bC^* \}}{|H||\Aut(H)|} \prod_{i=1}^{\infty} (1-p^{1-2i}).
\end{equation*}
\end{thm}

Let $S_p$ be the set of all finite abelian $p$-groups of the form $G \times G$ for some finite abelian $p$-group $G$. For $H \in S_p$, let $\mr{Sp}(H)$ denote the subgroup of $\Aut(H)$ consisting of automorphisms that preserve a fixed nondegenerate alternating bilinear pairing on $H$. The cardinality of $\mr{Sp}(H)$ is independent of the choice of such a pairing.

\begin{thm} \label{thm_alt}
(\cite[Theorem 1.13]{NW25}) Let $\al \in (0,1/2]$ and let $A(n) \in \mr{Alt}_n(\zp)$ be an $\al$-balanced random matrix for each $n \ge 1$. 
Then for every finite abelian $p$-group $H$,
\begin{equation*}
\lim_{n \to \infty} \bP(\cok(A(2n)) \cong H) 
= \left\{\begin{matrix}
\frac{|H|}{|\mr{Sp}(H)|} \prod_{i=1}^{\infty} (1-p^{1-2i}) & \text{if } H \in S_p,\\
0 & \text{otherwise}
\end{matrix}\right.
\end{equation*}
and 
\begin{equation*}
\lim_{n \to \infty} \bP(\cok(A(2n+1)) \cong \zp \times H)
= \left\{\begin{matrix}
\frac{1}{|\mr{Sp}(H)|} \prod_{i=2}^{\infty} (1-p^{1-2i}) & \text{if } H \in S_p,\\
0 & \text{otherwise}.
\end{matrix}\right.
\end{equation*}
\end{thm}

By reduction modulo $p$, one obtains the corresponding universality results for the rank distribution of random non-symmetric, symmetric, and alternating matrices over $\fp$.

In light of the above results, it is natural to ask how far the universality phenomenon can be extended to the optimal $\al_n$-balanced setting. In fact, related questions were explicitly raised by Wood \cite[Open Problems 3.3 and 3.10]{Woo22}. A first observation in this direction is that universality fails at the critical scale $\al_n = \frac{\log n}{n}$ (see Section \ref{Sec5}). On the other hand, in the non-symmetric matrix case over a finite field, the second author \cite{Lee25} established a sharp threshold for universality as follows.

\begin{thm} \label{thm_nonsym_Fp}
(\cite[Theorem 1.3]{Lee25}) Let $c>1$ be a constant, $\al_n = \frac{c \log n}{n}$, and let $A(n)$ be an $\al_n$-balanced random matrix in $\M_n(\fp)$ for each $n \ge 1$. Then for every nonnegative integer $k$,
\begin{equation*} 
\lim_{n \to \infty} \bP(\rank(A(n)) = n-k)
= p^{-k^2} \frac{\prod_{i=1}^{\infty} (1-p^{-i})}{\prod_{i=1}^{k} (1-p^{-i})^2}.
\end{equation*}
\end{thm}

Our goal is to prove a sharp universality result for various random matrix models over $\zp$: non-symmetric matrices, symmetric matrices and alternating matrices. The following theorem, which is the main result of this paper, generalizes Theorems \ref{thm_nonsym}, \ref{thm_sym} and \ref{thm_alt}.

\begin{thm} \label{main thm}
Let $c>1$ be a constant, $\al_n = \frac{c \log n}{n}$ and let $H$ be a finite abelian $p$-group. Suppose that $A(n)$ is an $\al_n$-balanced random matrix in $\mf{M}_n(\zp)$ for each $n \ge 1$, where $\mf{M}_n$ is one of $\M_n$, $\mr{Sym}_n$ or $\mr{Alt}_n$.
\begin{enumerate}[label=(\alph*)]
    \item (Non-symmetric case) If $\mf{M}_n = \M_n$, then
    \begin{equation} \label{eq_main_nonsym}
    \lim_{n \to \infty} \bP(\cok(A(n)) \cong H) 
    = \frac{1}{|\Aut(H)|} \prod_{i=1}^{\infty} (1-p^{-i}).
    \end{equation}
    
    \item (Symmetric case) If $\mf{M}_n = \mr{Sym}_n$, then
    \begin{equation} \label{eq_main_sym}
    \lim_{n \to \infty} \bP(\cok(A(n)) \cong H) 
    = \frac{\# \{ \text{symmetric, bilinear, perfect } \phi : H \times H \to \bC^* \}}{|H||\Aut(H)|} \prod_{i=1}^{\infty} (1-p^{1-2i}).
    \end{equation}

    \item (Alternating case) If $\mf{M}_n = \mr{Alt}_n$, then
    \begin{equation} \label{eq_main_alt_even}
    \lim_{n \to \infty} \bP(\cok(A(2n)) \cong H) 
    = \left\{\begin{matrix}
    \frac{|H|}{|\mr{Sp}(H)|} \prod_{i=1}^{\infty} (1-p^{1-2i}) & \text{if } H \in S_p,\\
    0 & \text{otherwise}
    \end{matrix}\right.
    \end{equation}
    and 
    \begin{equation} \label{eq_main_alt_odd}
    \lim_{n \to \infty} \bP(\cok(A(2n+1)) \cong \zp \times H)
    = \left\{\begin{matrix}
    \frac{1}{|\mr{Sp}(H)|} \prod_{i=2}^{\infty} (1-p^{1-2i}) & \text{if } H \in S_p,\\
    0 & \text{otherwise}.
    \end{matrix}\right.
    \end{equation}
\end{enumerate}
\end{thm}

By reduction modulo $p$, we obtain the sharp universality result for random symmetric and alternating matrices over $\fp$. See \cite[(20), (23), (26)]{FG15} or \cite[Theorems 1.10 and 1.13]{NW25} for the formulas of the rank distribution of random symmetric and alternating matrices over $\fp$.

\begin{cor} \label{main thm_fp}
Let $c>1$ be a constant and $\al_n = \frac{c \log n}{n}$.
\begin{enumerate}[label=(\alph*)]
    \item Suppose that $A(n)$ is an $\al_n$-balanced random matrix in $\mr{Sym}_n(\fp)$ for each $n \ge 1$. Then for every nonnegative integer $k$,
    \begin{equation} \label{eq_sym_fp}
    \lim_{n \to \infty} \bP(\rank(A(n)) = n-k) 
    = p^{-\frac{k(k+1)}{2}} \frac{\prod_{i=k+1}^{\infty} (1-p^{-i})}{\prod_{i=1}^{\infty}(1-p^{-2i})}.
    \end{equation}

    \item Suppose that $A(n)$ is an $\al_n$-balanced random matrix in $\mr{Alt}_n(\fp)$ for each $n \ge 1$. Then for every nonnegative integer $k$,
    \begin{equation} \label{eq_alt_fp}
    \begin{split}
    \lim_{n \to \infty} \bP(\rank(A(2n)) = 2n-2k) 
    & = p^{-k(2k-1)}\frac{\prod_{i=k}^{\infty} (1-p^{-2i-1})}{\prod_{i=1}^{k}(1-p^{-2i})}, \\
    \lim_{n \to \infty} \bP(\rank(A(2n+1)) =2n-2k) 
    & = p^{-k(2k+1)}\frac{\prod_{i=k+1}^{\infty} (1-p^{-2i-1})}{\prod_{i=1}^{k}(1-p^{-2i})}.
    \end{split}
    \end{equation}
\end{enumerate}
\end{cor}

Let $\Gamma \in G(n,q)$ be an Erdős–Rényi random graph on $n$ vertices, where each edge is included independently with probability $0<q<1$. Let $S_{\Gamma}$ denote the sandpile group (Jacobian) of $\Gamma$, and let $S_{\Gamma, p}$ be the Sylow $p$-subgroup of $S_{\Gamma}$. 
As an application of universality for random symmetric matrices over $\zp$ (Theorem \ref{thm_sym}), Wood \cite[Theorem 1.1]{Woo17} determined the limiting distribution of $S_{\Gamma, p}$ for $\Gamma \in G(n,q)$. We generalize this result to a broader class of Erdős–Rényi random graph sequences. 

\begin{thm} \label{main thm_graph}
Let $c>1$ be a constant, $\al_n = \frac{c \log n}{n} \le 1/2$, $\be_n \in [\al_n, 1-\al_n]$ and let $\Gamma(n) \in G(n, \be_n)$ be an Erdős–Rényi random graph. Then for every finite abelian $p$-group $H$,
\begin{equation} \label{eq_main_graph}
\lim_{n \to \infty} \bP(S_{\Gamma(n), p} \cong H) = \frac{\# \{ \text{symmetric, bilinear, perfect } \phi : H \times H \to \bC^* \}}{|H||\Aut(H)|} \prod_{i=1}^{\infty} (1-p^{1-2i}).
\end{equation}
\end{thm}

\begin{rmk} \label{rmk_graph threshold}
It is well-known that $\be_n = \frac{\log n}{n}$ is a sharp threshold for the connectivity of the random graph $\Gamma \in G(n,\be_n)$. Indeed, Erdős and Rényi (see \cite[Theorem 7.3]{Bol01}) proved that if $\be_n = \frac{\log n + c}{n}$ for a constant $c \in \bR$, then 
$$
\lim_{n \to \infty} \bP(\Gamma \text{ is connected}) = e^{-e^{-c}}.
$$
Our result shows that the distribution of the sandpile group of Erdős–Rényi random graph also has a sharp threshold at $\frac{\log n}{n}$ since a graph is connected if and only if its sandpile group is finite.
\end{rmk}

We write $\Hom(A,B)$ (resp. $\Sur(A,B)$) for the set of all homomorphisms (resp. surjective homomorphisms) from $A$ to $B$.
For a finite abelian $p$-group $G$, the exterior power $\wedge^2 G$ is the quotient of $G \otimes G$ by the subgroup generated by elements of the form $g \otimes g$. Similarly, the symmetric power $\mr{Sym}^2 G$ is the quotient of $G \otimes G$ by the subgroup generated by elements of the form $g_1 \otimes g_2 - g_2 \otimes g_1$. Note that when $G$ is a $p$-group of type $\lambda = (\lambda_1 \ge \cdots \ge \lambda_r)$ (i.e. $G \cong \bZ/p^{\lambda_1}\bZ \times \cdots \times \bZ/p^{\lambda_r}\bZ$) and $\lambda' = (\lambda'_1 \ge \cdots \ge \lambda'_{\lambda_1})$ is the conjugate partition of $\lambda$, 
$$
\lt \wedge^2 G \rt = p^{\sum_{j=1}^{\lambda_1} \frac{\lambda'_j(\lambda'_j-1)}{2}}
$$
and
$$
\lt \mr{Sym}^2 G \rt = p^{\sum_{j=1}^{\lambda_1} \frac{\lambda'_j(\lambda'_j+1)}{2}}.
$$

One of the key ingredients of the proofs of Theorems \ref{thm_nonsym}, \ref{thm_sym} and \ref{thm_alt} is the use of (surjective) moments of random groups. Let $(X_n)_{n \ge 1}$ and $(Y_n)_{n \ge 1}$ be sequences of random finitely generated $\zp$-modules. Suppose that
$$
\lim_{n \to \infty} \bE(\# \Sur(X_n, G)) = \lim_{n \to \infty} \bE(\# \Sur(Y_n, G)) = M_G
$$
for all finite abelian $p$-group $G$ and the quantities $M_G$ do not grow too rapidly. Then $X_n$ and $Y_n$ have the same limiting distribution. 
The proofs of our main results (Theorems \ref{main thm} and \ref{main thm_graph}) are also based on this moment method. More precisely, they follow from the next theorems together with \cite[Theorem 3.1]{Woo19} and \cite[Theorem 4.1]{NW25}.

\begin{thm} \label{main thm_moment}
Let $c>1$ be a constant and $\al_n = \frac{c \log n}{n}$. Suppose that $A(n)$ is an $\al_n$-balanced random matrix in $\mf{M}_n(\zp)$ for each $n \ge 1$, where $\mf{M}_n$ is one of $\M_n$, $\mr{Sym}_n$ or $\mr{Alt}_n$. Then for every finite abelian $p$-group $G$,
$$
\lim_{n \to \infty} \bE(\# \Sur(\cok(A(n)), G)) = \left\{\begin{matrix}
1 & \text{if} \quad \mf{M}_n = \M_n, \\
\lt \wedge^2 G \rt & \text{if} \quad \mf{M}_n = \mr{Sym}_n, \\
\lt \mr{Sym}^2 G \rt & \text{if} \quad \mf{M}_n = \mr{Alt}_n.
\end{matrix}\right.
$$
\end{thm}

\begin{thm} \label{main thm_moment graph}
Let $c>1$ be a constant, $\al_n = \frac{c \log n}{n} \le 1/2$, $\be_n \in [\al_n, 1-\al_n]$ and let $\Gamma(n) \in G(n, \be_n)$ be an Erdős–Rényi random graph. Then for every finite abelian $p$-group $G$,
$$
\lim_{n \to \infty} \bE(\# \Sur(S_{\Gamma(n), p}, G)) = 
\lt \wedge^2 G \rt.
$$
\end{thm}

Previous results, such as those of Wood \cite{Woo17}, \cite{Woo19} and Nguyen--Wood \cite{NW25} (as well as most subsequent works following their approach), analyze $F \in \Sur(R^n,G)$ by decomposing the space into “code” and “non-code” contributions. Within this framework, the main term arises from the code part, while the non-code part is treated as an error term and bounded accordingly.
However, the moment expressions are given by sums of $\ze^{C(F(A(n)))}$ where $C \in \Hom(\Hom(R^n, G), R)$. This indicates that the interaction between $C$ and $F$ plays a decisive role, and that a refined analysis of their relationship is essential in order to minimize the resulting error terms. Our approach is based precisely on exploiting this structural interaction, which allows us to sharpen the error analysis and ultimately obtain an optimal threshold.

This perspective motivates a change in the choice of moments.
Most existing universality results (including Theorems \ref{thm_nonsym}, \ref{thm_sym} and \ref{thm_alt}) are proved by computing \emph{Sur-moments}. In contrast, we work with the \emph{Hom-moments}
$$
\bE(\#\Hom(\cok(A(n)),G)),
$$
which are more tractable in our argument exploiting the relationship between $F$ and $C$. 
Note that this is purely a methodological choice: Hom-moments and Sur-moments are equivalent by the relations
$$
\#\Hom(X,G)=\sum_{K\le G}\#\Sur(X,K) \text{ and } \#\Sur(X,G)=\sum_{K\le G}\mu(K,G)\#\Hom(X,K),
$$
where $\mu(K,G)$ denotes the Möbius function of the subgroup lattice. 

In addition, our approach yields a unified proof of universality for several classes of random $p$-adic matrices, thereby addressing \cite[Open Problem 3.11]{Woo22}. In each model, the main difficulty is essentially the same: one must control the error terms in the relevant moment computations. Rather than treating the different cases separately, we express the Hom-moments in a common general form and establish error bounds that apply uniformly across all cases considered in our paper. 

We expect that our method can be applied to determine the sharp threshold for the existing universality results on cokernels of random $p$-adic matrices, including the joint distribution of multiple cokernels (\cite[Corollary 1.8]{CY23}, \cite[Theorem 1.1]{NVP24}) and the distribution of cokernels of $p$-adic Hermitian matrices \cite[Theorem 1.6]{Lee23}. A more challenging problem is to determine the sharp threshold for global universality results. Nguyen and Wood \cite[Theorem 2.4]{NW22} proved universality for random integral matrices whose entries are i.i.d. $n^{-1+\epsilon}$-balanced random integers, for a fixed $\epsilon \in (0,1)$. The same authors \cite[Theorems 1.12 and 1.16]{NW25} later established global universality for random symmetric and alternating matrices whose (independent) entries are i.i.d. $\al$-balanced random integers for a fixed constant $\al>0$. It is plausible that our approach could also improve such global universality results, although doing so would likely require corresponding refinements of the global universality arguments themselves.

The paper is organized as follows. In Section \ref{Sec2}, we derive explicit moment formulas for each matrix model (non-symmetric, symmetric and alternating), and then reformulate them in a unified abstract framework that covers all three symmetry types, as well as the random graph model. In Section \ref{Sec3}, we compute the main term of the Hom-moments in each matrix/graph model; in particular, we relate the relevant contributions to the number of maximal isotropic subgroups associated to a symmetric or alternating pairing on a finite abelian group.

Section \ref{Sec4} is the technical heart of the paper, where we prove that the remaining contributions form an error term that tends to $0$ as $n\to\infty$. We obtain this via a single argument in the unified setting rather than by separate case-by-case estimates. In Section \ref{Sec5}, we prove the sharpness of the $\frac{\log n}{n}$ threshold in our main results by showing that universality fails at the critical scale $\frac{\log n}{n}$. Finally, in Section \ref{Sec6}, we outline how to extend our main results to the $\mc{P}$-primary part of the cokernel of a random integral matrix, where $\mc{P}$ is a finite set of primes.

\section{The moment} \label{Sec2}

\subsection{Setting} \label{Sub21}
The following notation and conventions will be used throughout the paper.
\begin{itemize}
\item The matrix space $\mf{M}_n$ denotes either $\M_n$, $\mr{Sym}_n$, or $\mr{Alt}_n$.
\item Let $c>1$ be a real number and set $\al_n:=\frac{c\log n}{n}$. For a positive integer $m$, denote $[m] := \{ 1, 2, \ldots, m \}$.
\item Let $A(n)$ be an $\al_n$-balanced random matrix in $\mf{M}_n(\zp)$. 
\item Let $R := \bZ / p^d \bZ$ for a positive integer $d$ and $R^\times$ be the unit group of $R$. Let $\ze := \exp\left(\frac{2\pi i}{p^d}\right) \in \bC$.
\item Let $G = \bZ/p^{d_1}\bZ \times \cdots \times \bZ/p^{d_r}\bZ$ be a finite abelian $p$-group with $d\ge d_1 \ge d_2 \ge \cdots \ge d_r \ge 1$. Then $G$ is naturally an $R$-module. 
\item Define an $R$-bilinear pairing $\cdot : G \times G \to R$ by
$$
(x_1 + p^{d_1}\bZ, \ldots, x_r+p^{d_r}\bZ) \cdot (y_1 + p^{d_1}\bZ, \ldots, y_r+p^{d_r}\bZ) := \sum_{i=1}^{r} p^{d-d_i}x_iy_i + p^d \bZ.
$$
This pairing is nondegenerate; hence the map $\varphi : G \to \Hom(G,R)$ ($\varphi(g)(h)=g \cdot h$) is an isomorphism.
\item Let $\mr{Ran}(R, \al_n)$ denote the set of all $\al_n$-balanced random elements of $R$.
\item We write $\bP$ for probability and $\bE$ for expected value. We write $f(n)=\Theta(g(n))$ if there exist positive constants $c_1$ and $c_2$ such that $c_1g(n) \le f(n) \le c_2g(n)$ for all sufficiently large $n$.
\end{itemize}
The following constants will be used frequently in Section \ref{Sec4}:
\begin{itemize}
\item Fix constants $c_1, c_2, c_3$ satisfying $1<c_3<c_2<c_1<c$. For example, one may take $c_1=\frac{1+c}{2}$, $c_2=\frac{1+c_1}{2}$ and $c_3=\frac{1+c_2}{2}$.
\item Let $\ga, \ga_2>0$ be sufficiently small and let $\ga_1>0$ be sufficiently large constants. The constant $\ga$ will be specified in the proof of Proposition \ref{prop_case2}, while $\ga_1$ and $\ga_2$ will be specified in the proofs of Propositions \ref{prop_case3a} and \ref{prop_case3c}, respectively.
\end{itemize}

\subsection{Moment formula for each case} \label{Sub22}
Let $A(n)$ be an $\al_n$-balanced random matrix in $\mf{M}_n(\bZ_p)$, and denote by $\ol{A}(n)$ the reduction of $A(n)$ modulo $p^d$. Let $v_1, \ldots, v_n$ be the standard basis of $R^n$ and $G$ be a finite abelian $p$-group such that $p^d G=0$. If we regard $A(n)$ (resp. $\ol{A}(n)$) as a linear map from $\zp^n$ to $\zp^n$ (resp. $R^n$ to $R^n$), then the Hom-moment is given by
\begin{align*}
\bE(\# \Hom(\cok(A(n)), G)) 
&=\sum_{F \in \Hom(\bZ_p^n, G)} \bP(F(A(n))=0)
\\&=\sum_{F\in \Hom(R^n, G)} \bP(F(\ol{A}(n))=0)
\\&=\frac{1}{\lt G\rt^n} \sum_{\substack{F\in\Hom(R^n, G)\\C\in \Hom(\Hom(R^n,G), R)}}\bE\left(\ze^{C(F(\ol{A}(n)))}\right)
\\&=\frac{1}{\lt G\rt^n} \sum_{\substack{F\in\Hom(R^n, G) \\ C_1, \ldots, C_n \in \Hom(G,R)}} \bE\left(\ze^{\sum_{k,l\in[n]} C_l(\ol{A}(n)_{kl}F(v_k))}\right)
\\&=\frac{1}{\lt G\rt^n} \sum_{\substack{g_1,\cdots,g_n\in G\\ h_1,\cdots,h_n\in G}}\bE\left(\ze^{\sum_{k,l\in[n]}\ol{A}(n)_{kl}(g_k \cdot h_l)}\right).
\end{align*}
Here the third equality uses the discrete Fourier transform on the finite abelian group $\Hom(R^n, G)$, and we put $g_k=F(v_k)$ and $C_l(g)=g \cdot h_l$ (equivalently, $\varphi(h_l)=C_l$ for the isomorphism $\varphi : G \to \Hom(G,R)$ defined in Section \ref{Sub21}) in the last equality. 
For each matrix model $\mf{M}_n$, the Hom-moment $\bE(\# \Hom(\cok(A(n)), G))$ is given by
\begin{equation} \label{eq: the moment formula}
\begin{cases}
\displaystyle
\frac{1}{\lt G\rt^n} \sum_{\substack{g_1, \ldots, g_n \in G\\ h_1, \ldots, h_n \in G}} \prod_{k,l \in [n]} \bE\left(\ze^{\ol{A}(n)_{kl}(g_k \cdot h_l)}\right) & \text{if $\mf{M}_n=\M_n$},\\
\displaystyle
\frac{1}{\lt G\rt^n} \sum_{\substack{g_1, \ldots, g_n \in G\\ h_1, \ldots, h_n \in G}} \prod_{k \in [n]}  \bE\left(\ze^{\ol{A}(n)_{kk}(g_k\cdot h_k)}\right)\prod_{1\le k<l \le n} \bE\left(\ze^{\ol{A}(n)_{kl}(g_k \cdot h_l + g_l\cdot h_k)}\right) & \text{if $\mf{M}_n=\mr{Sym}_n$}, \\
\displaystyle
\frac{1}{\lt G\rt^n} \sum_{\substack{g_1, \ldots, g_n \in G\\ h_1, \ldots, h_n \in G}}\prod_{1\le k<l \le n} \bE\left(\ze^{\ol{A}(n)_{kl}(g_k \cdot h_l - g_l\cdot h_k)}\right) & \text{if $\mf{M}_n=\mr{Alt}_n$}.
\end{cases}
\end{equation}

Now we discuss the moment for the sandpile group of an Erdős–Rényi random graph.
Let $\be_n\in[\al_n, 1- \al_n]$, and let $\Gamma \in G(n, \be_n)$ be an Erdős–Rényi random graph on $n$ vertices with independent edge probability $\be_n$. Let $S(n)$ be the sandpile group of $\Gamma \in G(n, \be_n)$. 
Following the proof of \cite[Theorem 6.2]{Woo17}, we have
$$
\bE\left(\# \Hom\left(S(n), G\right)\right) = \lt G\rt^{-1}|R|^{n}\sum_{F \in \Hom(R^n,G)} \bP\left(\tilde{F}X(n) =0\right).
$$
Here $X(n)$ is a random matrix in $\mr{Sym}_n(R)$ such that
$X(n)_{kk}$ is uniform in $R$, 
$$
\bP(X(n)_{kl} = r) = \bP(X(n)_{lk} = r) = \begin{cases}
1-\be_n & \text{if $r=0$} \\
\be_n & \text{if $r=1$}
\end{cases}
$$
for all $k \ne l \in [n]$ and $\tilde{F}:=F\oplus\sigma\in \Hom(R^n, G\oplus R)$ with an augmentation map $\sigma:R^n\to R$ sending $(a_1,\ldots,a_n) \mapsto \sum_{i=1}^{n} a_i$. 
Following the computation of the Hom-moment of $\cok(A(n))$, we have
\begin{align*}
& \bE\left(\# \Hom\left(S(n), G\right)\right) \\
= \, & \frac{|R|^{n}}{\lt G\rt}\cdot\frac{1}{\lt G\oplus R\rt^n}\sum_{\substack{F \in \Hom(R^n,G) \\\tilde{C}\in\Hom(\Hom(R^n,G\oplus R),R)}} \bE\left(\ze^{\tilde{C}(\tilde{F}(X(n)))}\right) \\
= \, & \frac{1}{\lt G\rt^{n+1}}\sum_{\substack{F \in \Hom(R^n,G)\\(C_1,C'_1), \ldots, (C_n,C'_n) \in \Hom(G\oplus R,R)}} \bE\left(\ze^{\sum_{k,l\in[n]}(C_l(X(n)_{kl}F(v_k))+C_l'(X(n)_{kl}))}\right) \\
= \, & \frac{1}{\lt G\rt^{n+1}}\sum_{\substack{g_1,\cdots,g_n\in G\\h_1,\cdots,h_n\in G}}\sum_{a_1,\cdots,a_n\in R}\bE\left(\ze^{\sum_{k,l\in[n]}X(n)_{kl}(g_k\cdot h_l+a_l)}\right) \\
= \, & \frac{1}{\lt G\rt^{n+1}}\sum_{\substack{g_1,\cdots,g_n\in G\\h_1,\cdots,h_n\in G}}\sum_{a_1,\cdots,a_n\in R}\prod_{k\in[n]}\bE\left(\ze^{X(n)_{kk}(g_k\cdot h_k+a_k)}\right)\prod_{1\le k<l\le n}\bE\left(\ze^{X(n)_{kl}(g_k\cdot h_l+g_l\cdot h_k+a_k+a_l)}\right).
\end{align*}
Since $X(n)_{kk}$ follows the uniform measure in $R$, 
$$\bE\left(\ze^{X(n)_{kk}(g_k\cdot h_k + a_k)}\right)
= \begin{cases}
1 & \text{if $a_k = -g_k\cdot h_k $,} \\
0 & \text{if $a_k \neq -g_k\cdot h_k $.}
\end{cases}
$$
Thus the Hom-moment $\bE\left(\# \Hom\left(S(n), G\right)\right)$ can be written as
\begin{equation}
\label{eq: the moment formula for graph}
\begin{split}
& \frac{1}{\lt G\rt^{n+1}} \sum_{\substack{g_1, \ldots, g_n \in G \\ h_1, \ldots, h_n \in G}}
\prod_{1\le k < l \le n}\bE\left(\ze^{X(n)_{kl}(g_k\cdot h_l+g_l\cdot h_k-g_k\cdot h_k-g_l\cdot h_l)}\right) \\
= \, & \frac{1}{\lt G\rt^{n+1}} \sum_{\substack{g_1, \ldots, g_n \in G \\ h_1, \ldots, h_n \in G}}
\prod_{1\le k < l \le n}\bE\left(\ze^{-X(n)_{kl}(g_k-g_l)\cdot (h_k  -h_l)}\right).
\end{split}
\end{equation}

\subsection{A general formula} \label{Sub23}
In this section, we provide a general formula that allows us to treat the Hom-moments appearing in \eqref{eq: the moment formula} and \eqref{eq: the moment formula for graph} simultaneously. First we provide some notations.
\begin{itemize}
    \item Let $\Delta:=\{(k,k):k\in[n]\}$ and let $\Omega$ be either $\Omega_1 := \{ (k,l) \in [n] \times [n] : k \ne l \}$ or $\Omega_2 := \{ (k,l) \in [n] \times [n] : k < l \}$. 
    
    \item $U : G^2 \to R$ is a function and $B:G^2\times G^2\to R$ is a (not necessarily bilinear) pairing that satisfies the following conditions:
    \begin{enumerate}[label = (\roman*)]
    \item The maximal possible cardinality of an isotropic subset of $G^2$ is $\lvert G\rvert$. (We say a subset $W \subseteq G^2$ is \emph{isotropic} if $B(\mb{g},\mb{g}')=0$ for all $\mb{g},\mb{g}'\in W$.) \label{bilinear_cond1}
    \item If an isotropic subset $W \subseteq G^2$ has exactly $\lt G\rt$ elements, then it is a coset $\mb{g}'+H$ for some subgroup $H\le G^2$ and $\mb{g}' \in G^2$. Moreover, for every $\mb{g}\notin W$, either $B(\mb{g}'+*,\mb{g})$ or  $B(\mb{g},\mb{g}'+*)$ is a non-constant affine map on $H$.\label{bilinear_cond2}
    \item If $\Omega=\Omega_2$, then $B$ is either symmetric or alternating. \label{bilinear_cond3}
    \end{enumerate}

    \item The \emph{trivial part} of $G^{2n}$ (with respect to given $U$ and $B$) is the set 
    $$
    G_{tr}:=\{(\mb{g}_1,\ldots,\mb{g}_n)\in G^{2n}:U(\mb{g}_k)=B(\mb{g}_k,\mb{g}_l)=0\text{ for all } k,l\in[n]\}.
    $$
\end{itemize}

Now for each $n \ge 1$, consider the function 
$$
f:\mr{Ran}(R,\al_n)^{\Omega\cup\Delta}\to\bC
$$ 
defined by 
\begin{align*}
f((x_{k,l,n}))&:=\frac{1}{\lt G\rt^n} \sum_{\mb{g}_1,\cdots,\mb{g}_n\in G^2}\prod_{k\in[n]}\bE\left(\ze^{U(\mb{g}_k)x_{k,k,n}}\right)\cdot\prod_{(k,l)\in \Omega}\bE\left(\ze^{B(\mb{g}_k,\mb{g}_l)x_{k,l,n}}\right)
\\&=\frac{\lt G_{tr}\rt}{\lt G\rt^n}+\frac{1}{\lt G\rt^n}\sum_{(\mb{g}_1,\cdots,\mb{g}_n)\notin G_{tr}}\prod_{k\in[n]}\bE\left(\ze^{U(\mb{g}_k)x_{k,k,n}}\right)\cdot\prod_{(k,l)\in \Omega}\bE\left(\ze^{B(\mb{g}_k,\mb{g}_l)x_{k,l,n}}\right).
\end{align*}
Write $\mb{g}_k=(g_k,h_k)$ ($g_k, h_k \in G$). 
For each random matrix model, the Hom-moment can be written in the form 
$f((\ol{A}(n)_{k,l}))$ for the following choices of $\Omega$, $U$, and $B$ in cases (1)--(3);  for the random graph model, the Hom-moment is equal to $\frac{f((X(n)_{k,l}))}{\lt G \rt}$ for $\Omega$, $U$, and $B$ as in case (4). 
\begin{enumerate}
    \item Non-symmetric: $\Omega=\Omega_1$, $U(\mb{g}_k) = U_1(\mb{g}_k) := g_k \cdot h_k$ and $B(\mb{g}_k, \mb{g}_l) = B_1(\mb{g}_k, \mb{g}_l) := g_k \cdot h_l$.
    \item Symmetric: $\Omega=\Omega_2$, $U(\mb{g}_k) = U_2(\mb{g}_k) := g_k \cdot h_k$ and $B(\mb{g}_k, \mb{g}_l) = B_2(\mb{g}_k, \mb{g}_l) := g_k \cdot h_l + g_l \cdot h_k$.
    \item Alternating: $\Omega=\Omega_2$, $U(\mb{g}_k) = U_3(\mb{g}_k) := 0$ and $B(\mb{g}_k, \mb{g}_l) = B_3(\mb{g}_k, \mb{g}_l) := g_k \cdot h_l - g_l \cdot h_k$.
    \item Graph: $\Omega=\Omega_2$, $U(\mb{g}_k) = U_4(\mb{g}_k) := 0$ and $B(\mb{g}_k, \mb{g}_l) = B_4(\mb{g}_k, \mb{g}_l) := -(g_k-g_l) \cdot (h_k-h_l)$.
\end{enumerate}
Define
$$
\mr{Iso}^{U,B}_G := \left\{W \le G^2 : U(\mb{g}) = 0 \text{ and } B(\mb{g}, \mb{g}') = 0 \text{ for all $\mb{g},\mb{g}' \in W$}\right\},
$$
the set of isotropic subgroups of \emph{$U$-vanishing type}. 
\begin{lem}
In each of the cases (1)--(4), the pairing $B$ satisfies conditions \ref{bilinear_cond1}, \ref{bilinear_cond2} and \ref{bilinear_cond3}. (The condition \ref{bilinear_cond3} is vacuous in case (1).)
\end{lem}

\begin{proof}
It is straightforward to verify that the lemma holds in cases (1)--(3); therefore, we provide a proof only for case (4).
Let $W$ be an isotropic subset of $G^2$ of maximal possible size. By Lemma \ref{lem: graph lemma coset}, we have $W =\mb{g}'+H$ for some $\mb{g}' \in G^2$ and $H \in \mr{Iso}^{U_2,B_2}_G$ with $|H|=|G|$. This proves \ref{bilinear_cond1} for case (4). For \ref{bilinear_cond2}, let $\mb{g}\notin W$ and write $\mb{g}-\mb{g}'= (a_1,b_1) \notin H$. Then for every $(a,b) \in H$, $a \cdot b = 0$ so 
$$
B_4(\mb{g}, \mb{g}'+(a,b)) = B_4(\mb{g}-\mb{g}', (a,b)) = -a_1\cdot b_1 + B_2((a_1,b_1), (a,b)).
$$
Since $B_2$ is bilinear, $B_4(\mb{g}, \mb{g}'+*)$ (and similarly $B_4(\mb{g}'+*, \mb{g})$) is an affine map on $H$. Furthermore, $B_4(\mb{g}, \mb{g}'+*)$ (and similarly $B_4(\mb{g}'+*, \mb{g})$) is a non-constant map since $(a_1, b_1) \notin H= H^{\perp_{B_2}}$. Finally, $B_4$ is symmetric so it satisfies \ref{bilinear_cond3}.
\end{proof}

The following theorem (proved in Section \ref{Sec4}) shows that the Hom-moments are asymptotically determined by the trivial part. Combined with Proposition \ref{prop: main terms for random matrix} (resp. Proposition \ref{prop: main terms for graph}), along with the relation between Hom-moments and Sur-moments, this yields Theorem \ref{main thm_moment} (resp. Theorem \ref{main thm_moment graph}).

\begin{thm}\label{thm: general}
For any $(x_{k,l,n}) \in \mr{Ran}(R,\al_n)^{\Omega\cup\Delta}$, the value
$$
\lt f((x_{k,l,n}))-\frac{\lt G_{tr}\rt}{\lt G\rt^n}\rt=\frac{1}{\lt G\rt^n}\lt\sum_{(\mb{g}_1,\cdots,\mb{g}_n)\notin G_{tr}}\prod_{k\in[n]}\bE\left(\ze^{U(\mb{g}_k)x_{k,k,n}}\right)\cdot\prod_{(k,l)\in \Omega} \bE\left(\ze^{B(\mb{g}_k,\mb{g}_l)x_{k,l,n}}\right)\rt
$$
converges to 0 as $n\to\infty$. 
\end{thm}

\section{Main term of the moments} \label{Sec3}

In this section, we compute the main term $\frac{\lt G_{tr}\rt}{\lt G\rt^n}$ ($\frac{\lt G_{tr}\rt}{\lt G\rt^{n+1}}$ for random graphs) of each random model.

\begin{prop}
\label{prop: main terms for random matrix}
Let $U$, $B$ and $G_{tr}$ be as before. Then
$$
\lim_{n\to \infty}\frac{\lt G_{tr}\rt}{\lt G\rt^n} = \begin{cases}
\displaystyle
\sum_{H \le G} 1    & \text{if $\mf{M}_n=\M_n$}, \\
\displaystyle
\sum_{H \le G} |\wedge^2 H|    & \text{if $\mf{M}_n=\mr{Sym}_n$}, \\
\displaystyle
\sum_{H \le G} |\mr{Sym}^2 H|    & \text{if $\mf{M}_n=\mr{Alt}_n$}.
\end{cases}
$$
\end{prop}

\begin{proof}
Let $r(H) = \dim_{\fp} H/pH$ and write $\mb{g}_k=(g_k,h_k)$. Then
\begin{align} \label{number_gen}
\#\{(g_1,\cdots,g_n)\in H^n:\langle g_1,\cdots,g_n\rangle=H\}=\lt H\rt^n\cdot\prod_{i=1}^{r(H)}(1-p^{i-1-n})
\end{align}
whenever $n\ge r(H)$. 
If $\mf{M}_n=\M_n$, then
$$
(\mb{g}_1, \ldots, \mb{g}_n) \in G_{tr} \, \Longleftrightarrow \, g_k \cdot h_l = 0 \text{ for all $k,l \in [n]$}.
$$
For $H \le G$, let $H^\perp := \{g \in G : g\cdot h =0 \text{ for all $h \in H$}\}.$
By \eqref{number_gen}, we have
\begin{align*}
\lim_{n \to \infty} \frac{\lt G_{tr}\rt}{\lt G\rt^n}
& = \lim_{n \to \infty} \frac{1}{\lt G\rt^n}\sum_{H \le G}\sum_{\substack{g_1,\ldots,g_n \in H \\ \langle g_1,\ldots, g_n\rangle = H} } \sum_{h_1, \ldots, h_n \in H^\perp} 1\\
& = \lim_{n \to \infty} \frac{1}{\lt G\rt^n}\sum_{H \le G}\sum_{\substack{g_1,\ldots,g_n \in H \\ \langle g_1,\ldots, g_n\rangle = H} } \lt H^\perp\rt^n\\
& = \lim_{n \to \infty} \sum_{H\le G}\prod_{i=1}^{r(H)}(1-p^{i-1-n}) \\
& = \sum_{H\le G} 1.
\end{align*}
Now suppose that $\mf{M}_n = \mr{Sym}_n$ or $\mr{Alt}_n$. 
Again by \eqref{number_gen}, we have
\begin{align*}
\lim_{n \to \infty} \frac{\lt G_{tr}\rt}{\lt G\rt^n} 
&=\lim_{n \to \infty} \frac{1}{\lt G\rt^n}\sum_{W \in \mr{Iso}^{U,B}_G}\sum_{\substack{\mb{g}_1, \ldots, \mb{g}_n \in W \\\langle \mb{g}_1, \ldots, \mb{g}_n\rangle = W}} 1\\
&=\lim_{n \to \infty} \sum_{W \in \mr{Iso}^{U,B}_G}\left(\frac{\lt W\rt}{\lt G\rt}\right)^n\cdot\prod_{i=1}^{r(W)}(1-p^{i-1-n}) \\
& =\#\{W\in\mr{Iso}^{U,B}_G:\lt W\rt=\lt G\rt\},
\end{align*}
where last term is the number of maximal isotropic subgroups of $U$-vanishing type by assumption \ref{bilinear_cond1} in Section \ref{Sub23}. The proposition follows from the correspondences \eqref{eq: first bijection} and \eqref{eq: second bijection}. We remark that $U_2((g,h)) = P(g,h)$ where $P$ is a pairing defined in Section \ref{Sub31} and $U_3=0$.
\end{proof}

To compute the main term in the random graph case, we first describe a criterion for $(\mb{g}_1, \ldots, \mb{g}_n) \in G_{tr}$.

\begin{lem}
\label{lem: graph lemma coset}
Suppose that $(U,B)$ is given as in the random graph model. Then $(\mb{g}_1, \ldots, \mb{g}_n) \in G_{tr}$ if and only if there exists $W \in \mr{Iso}^{U_2,B_2}_G$ such that $\langle \mb{g}_1-\mb{g}_n, \ldots, \mb{g}_{n-1}-\mb{g}_n \rangle = W$.
\end{lem}

\begin{proof}
Write $\mb{g}_k-\mb{g}_n = \mb{h}_k = (a_k,b_k)$ for $a_k,b_k\in G$. If $\langle \mb{h}_1, \ldots, \mb{h}_{n-1} \rangle \in \mr{Iso}^{U_2,B_2}_G$, then
$$
B(\mb{g}_k, \mb{g}_l) 
= -(g_k-g_l)\cdot(h_k-h_l) 
= -(a_k-a_l)\cdot(b_k-b_l) 
= -U_2(\mb{h}_k) - U_2(\mb{h}_l) + B_2(\mb{h}_k, \mb{h}_l) = 0
$$
for all $k,l \in [n]$ so $(\mb{g}_1, \ldots, \mb{g}_n) \in G_{tr}$. 
Conversely, if $B(\mb{g}_k,\mb{g}_l)=0$ for all $k, l\in[n]$, then
$$
U_2(\mb{h}_k)=a_k \cdot b_k=(g_k-g_n)\cdot(h_k-h_n)=-B(\mb{g}_k,\mb{g}_n)=0
$$
and
$$
B_2(\mb{h}_k, \mb{h}_l)=a_k\cdot b_k+a_l\cdot b_l+B(\mb{g}_k,\mb{g}_l)=0
$$
so $W = \langle \mb{h}_1, \ldots, \mb{h}_{n-1} \rangle$ is an element of $\mr{Iso}^{U_2,B_2}_G$.
\end{proof}

\begin{prop}
\label{prop: main terms for graph}
Suppose that $(U,B)$ is given as in the random graph model. Then
$$
\lim_{n\to \infty}\frac{\lt G_{tr}\rt}{\lt G\rt^{n+1}} = \sum_{H \le G}|\wedge^2 H|.
$$
\end{prop}

\begin{proof}
By \eqref{number_gen} and Lemma \ref{lem: graph lemma coset}, we have
\begin{align*}
\lim_{n \to \infty} \frac{\lt G_{tr}\rt}{\lt G\rt^{n+1}}
& = \lim_{n \to \infty} \frac{1}{\lt G\rt^{n+1}} \sum_{W\in\mr{Iso}^{U_2,B_2}_G} \sum_{\substack{\mb{g}_n\in G^2\\\mb{g}_1,\cdots,\mb{g}_{n-1}\in\mb{g}_n+W\\\langle\mb{g}_1-\mb{g}_n,\cdots,\mb{g}_{n-1}-\mb{g}_n\rangle=W}}1\\
&=\lim_{n \to \infty} \frac{1}{\lt G\rt^{n+1}}\sum_{W\in\mr{Iso}^{U_2,B_2}_G}\sum_{\substack{\mb{g}\in G^2\\\mb{h}_1,\cdots,\mb{h}_{n-1}\in W\\\langle\mb{h}_1,\cdots,\mb{h}_{n-1}\rangle=W}}1\\
&=\lim_{n \to \infty} \sum_{W\in\mr{Iso}^{U_2,B_2}_G}\left(\frac{\lt W\rt}{\lt G\rt}\right)^{n-1}\cdot\prod_{i=1}^{r(W)}(1-p^{i-n}) \\
&= \#\{ W\in\mr{Iso}^{U_2,B_2}_G:\lt W\rt=\lt G\rt\}.
\end{align*}
The proposition follows from the correspondence \eqref{eq: second bijection}.
\end{proof}

\subsection{Bilinear pairings and maximal isotropic subgroups} \label{Sub31}
Let $P: G \times G \to \bQ/\bZ$ be a perfect bilinear pairing\footnote{In our application, $P$ will always be the pairing $\cdot$ defined in Section \ref{Sub21}. (More precisely, $P$ is the composition $G \times G \overset{\cdot}{\rightarrow} R \cong \frac{1}{p^d}\bZ/\bZ \hookrightarrow \bQ/\bZ$.) In this case, $B_{\mr{Sym}}=B_2$ and $B_{\mr{Alt}}=B_3$.}, and define two bilinear pairings $G^2 \times G^2 \to \bQ/\bZ$ by 
\begin{align*}
  B_{\mr{Alt}}((g,h),(g',h')) &:= P(g,h') - P(g',h), \\
  B_{\mr{Sym}}((g,h),(g',h')) &:= P(g,h') + P(g',h).
\end{align*}
Then $B_{\mr{Alt}}$ is alternating and $B_{\mr{Sym}}$ is symmetric.
Moreover, both pairings are perfect since $P$ is perfect.

\begin{prop} \label{prop3d}
There exist natural one-to-one correspondences:
\begin{align}
&\{\text{maximal isotropic subgroups of $(G^2,B_{\mr{Alt}})$}\}  \longleftrightarrow \bigsqcup_{H \le G} \Hom(\mr{Sym}^2(H), \bQ/\bZ) \label{eq: first bijection},\\
&\{\text{maximal isotropic subgroups of $(G^2,B_{\mr{Sym}})$ of $P$-vanishing type}\}  \longleftrightarrow \bigsqcup_{H \le G} \Hom(\wedge^2 H, \bQ/\bZ) \label{eq: second bijection}
\end{align}
with natural identifications
\begin{align*}
\Hom\left(\mr{Sym}^2(H), \bQ/\bZ\right) & = \{\text{symmetric bilinear pairings}~H\times H \to \bQ/\bZ\}, \\
\Hom\left(\wedge^2 H, \bQ/\bZ\right) & = \{\text{alternating bilinear pairings}~H\times H \to \bQ/\bZ\}.
\end{align*}
Here, a subgroup $W \le G^2$ is said to be of \emph{$P$-vanishing type} if $P(g,h) = 0$ for all $(g,h)\in W$.
\end{prop}

\begin{proof}
For a subgroup $H \le G$, define
$$
H^{\perp_P}:= \{g \in G: P(g, h) = 0 \text{ for all $ h \in H$}\}.
$$
Fix either $B=B_{\mr{Alt}}$ or $B_{\mr{Sym}}$, and let $W$ be a maximal isotropic subgroup of $(G^2,B)$. Let $\pi_i : G^2 \to G$ denote the projection onto the $i$-th factor. Set $H := \pi_2(W) \le G$ and $K := \pi_1(W \cap (G \times \{0 \} )) \le G$. Then we have a short exact sequence
$$
0 \to K \to W \xrightarrow{\pi_2} H \to 0.
$$
For every $h \in H$, there exists $k_h\in G$ such that $(k_h, h) \in W$ and hence 
$$
0=B((k,0), (k_h, h)) = P(k, h)\text{ for every }k \in K,
$$
which implies that $K \subseteq H^{\perp_P}$. Since the above short exact sequence gives 
$$
\lt K\rt\lt H\rt = \lt W\rt = \lt G\rt = \lt H^{\perp_P}\rt \lt H\rt,
$$
we conclude that $K = H^{\perp_P}$. Choose a section $s \colon H \to W$ of $\pi_2$ (which need not be a group homomorphism), and write $s(h) = (f(h), h)$ for some function $f: H \to G$. Define a pairing $\beta: H \times H \to \bQ/\bZ$ by 
$$
\beta(h_1, h_2):= P(f(h_1), h_2). 
$$
If we replace $f$ by another function arising from a different choice of section, then $f(h)$ changes by an element of $K=H^{\perp_P}$ so the map $\beta$ is well-defined. Moreover, for $h_1, h_2 \in H$ it is straightforward to see that
$$
0 = B(s(h_1), s(h_2)) = 
\begin{cases} 
\beta(h_1, h_2) - \beta(h_2, h_1) &\text{if } B=B_{\mr{Alt}}, \\
\beta(h_1, h_2) + \beta(h_2, h_1) &\text{if } B=B_{\mr{Sym}}.
\end{cases}
$$
Hence $\beta$ is symmetric (resp. skew-symmetric) if $B=B_{\mr{Alt}}$ (resp. $B=B_{\mr{Sym}}$). Moreover, $\beta$ is bilinear as it is already linear with respect to the second factor. If $B=B_{\mr{Sym}}$ and $W$ is of $P$-vanishing type, then 
$$
\beta(h, h) = P(f(h), h) = 0
$$
so $\beta$ is alternating. Therefore, we have constructed a map $W \mapsto (H, \beta)$ appearing in \eqref{eq: first bijection} and \eqref{eq: second bijection}.

Conversely, let $H \le G$ and let $\beta:H \times H \to \bQ/\bZ$ be a symmetric (resp. alternating) bilinear pairing when $B=B_{\mr{Alt}}$ (resp. $B=B_{\mr{Sym}}$).
Choose a function $f: H \to G$ such that a class of $f(h)$ modulo $H^{\perp_P}$ corresponds to the homomorphism $\beta(h, -) \in \Hom(H, \bQ/\bZ)$ under the isomorphism 
$$
G/H^{\perp_P} \cong \Hom(H, \bQ/\bZ)\quad\text{sending}\quad g+H^{\perp_P}\mapsto(\phi_g:h\mapsto P(g,h)).
$$
Equivalently, for each $h \in H$, the element $f(h)$ is chosen so that
$$
P(f(h),  h') = \beta(h, h') \text{ for all }h'\in H.
$$ 
Define
$$
W_{H, \beta} := \{(f(h)+k, h) \in G^2 : h \in H, k \in H^{\perp_P}\} \subseteq G^2.
$$
Then $W_{H, \beta}$ is independent of the choice of $f$ and
$$
B((f(h)+k, h), (f(h')+k', h')) = \left\{\begin{matrix}
\beta(h,h')-\beta(h',h)=0 & \text{if } B=B_{\mr{Alt}}, \\
\beta(h,h')+\beta(h',h)=0 & \text{if } B=B_{\mr{Sym}}
\end{matrix}\right.
$$
for all $h,h' \in H$ and $k,k' \in H^{\perp_P}$. 
Therefore, $W_{H, \beta}$ is isotropic in $(G^2,B)$ and $\lt W_{H, \beta}\rt = \lt H^{\perp_P}\rt \lt H\rt =\lt G\rt$ (so it is maximal subgroup of $G^2$). If $\beta$ is alternating, then 
$$
P(f(h)+k, h) = P(f(h), h) = \beta(h,h) = 0
$$
for all $h \in H$, $k\in H^{\perp_P}$ so $W_{H, \beta}$ is of $P$-vanishing type. The map $(H,\beta) \mapsto W_{H,\beta}$ is the inverse of the map $W \mapsto (H, \beta)$ constructed above.
\end{proof}

\section{Error term of the moments}\label{Sec4}

\subsection{Some lemmas}
First we provide several lemmas that will be used in the proof of 
Theorem \ref{thm: general}.

\begin{lem}\label{lem: extreme point}
Let $I$ and $J$ be finite sets, $b_{i,j} \in R$ for all $i \in I$, $j \in J$ and $\al\in(0,\frac{1}{2})$. Then for every $(x_j)_{j\in J}\in\mr{Ran}(R,\al)^J$, we have 
$$
\lt\sum_{i\in I}\prod_{j\in J}\bE\left(\ze^{b_{i,j}x_j}\right)\rt\le\sum_{i\in I}\prod_{j\in J}\lt\bE\left(\ze^{b_{i,j}z_j}\right)\rt=\sum_{i\in I}\prod_{j\in J}\lt1-\al+\al\ze^{b_{i,j}t_j}\rt,
$$
where for every $j \in J$, the random variable $z_j$ takes values in $R$ with $\bP(z_j=0)=1-\al$ and $\bP(z_j=t_j)=\al$ for some $t_j\in R^\times$. 
\end{lem}

\begin{proof}
Consider the set
$$
C_{p^d, \al} := \left\{ (s_1, \ldots, s_{p^d}) \in [0,1]^{p^d} : \sum_{i=1}^{p^d} s_i=1 \text{ and } \sum_{i \equiv a \pmod{p}} s_i \le 1-\al  \text{ for all } 0 \le a \le p-1 \right\}.
$$
It is convex and compact, and has a bijection to $\mr{Ran}(R, \al)$ via the map $(s_1, \ldots, s_{p^d}) \mapsto x$ defined by $\bP(x=a)=s_a$ for each $a \in R$. Since $\sum_{i\in I}\prod_{j\in J}\bE(\ze^{b_{i,j}x_j})$ is affine in each variable $x_j$ (viewed as an element of $C_{p^d,\al}$), Bauer's maximum principle implies that its absolute value attains a maximum at a collection of extreme points $z_j' \in C_{p^d,\al}$, one for each $j \in J$. Let $(s_1, \ldots, s_{p^d})$ be an extreme point of $C_{p^d,\al}$. 

If $s_i,s_j>0$ for two distinct indices with $i\equiv j \pmod p$, then for sufficiently small $\epsilon>0$, we may replace $(s_i, s_j)$ with $(s_i \pm \epsilon, s_j \mp \epsilon)$ while keeping all constraints satisfied, which contradicts extremality. Therefore, within each residue class modulo $p$, at most one coordinate can be positive. If $s_i\ge s_j\ge s_k>0$, then we may replace $(s_j, s_k)$ with $(s_j \pm \epsilon, s_k \mp \epsilon)$, which contradicts extremality. Finally, if $s_i, s_j>0$ for some $i \not\equiv j \pmod{p}$ and $s_t=0$ for all $t \ne i,j$, then extremality forces $(s_i, s_j)=(1-\al, \al)$ or $(\al, 1-\al)$.

By the assumption $\al<\frac{1}{2}$, for each $j \in J$ there exist $u_j,v_j\in R$ with $u_j\not\equiv v_j\pmod p$ such that $\bP(z_j'=u_j)=1-\al$ and $\bP(z_j'=v_j)=\al$. Then it follows that
$$
\sum_{i\in I}\prod_{j\in J}\lt\bE(\ze^{b_{i,j}z'_j})\rt=\sum_{i\in I}\prod_{j\in J}\lt(1-\al)\ze^{b_{i,j}u_j}+\al\ze^{b_{i,j}v_j}\rt=\sum_{i\in I}\prod_{j\in J}\lt1-\al+\al\ze^{b_{i,j}(v_j-u_j)}\rt
$$
and taking $t_j:=v_j-u_j$ finishes the proof.
\end{proof}

\begin{lem}\label{lem_ine1}
For every nonzero $r \in R$ and all sufficiently large $n$, 
$$
\lt 1-\al_n+\al_n\ze^r \rt
\le 1-\frac{2c_1\log n}{n}\sin^2\frac{\pi r}{p^d}
\le1-\frac{\log n}{n}\frac{8c_1}{p^{2d}}.
$$
\end{lem}

\begin{proof}
The proof is identical to that of \cite[Lemma 2.5]{Lee25}.
\end{proof}

\begin{lem}
\label{lem: sin sum}
Let $H$ be a finite abelian group and let $L : H \to \bZ/m\bZ$ be a non-constant affine map. Then
$$
\sum_{h \in H} \sin^2 \frac{\pi\cdot L(h)}{m} = \frac{\lt H\rt}{2}.
$$
In particular, if $W$ is an isotropic subset of $(G^2, B)$ with $|W| = |G|$, then for every $t \in R^\times$ and $\mb{g} \in G^2\setminus W$, 
$$
\sum_{\mb{g}' \in W} \sin^2\frac{\pi t B(\mb{g}, \mb{g}')}{p^d } = \frac{|G|}{2} \quad\text{or} \quad\sum_{\mb{g}' \in W} \sin^2\frac{\pi t B(\mb{g}', \mb{g})}{p^d } = \frac{|G|}{2}.
$$
If $B \in \{B_2, B_3, B_4\}$, then both equalities hold. 
\end{lem}

\begin{proof}
Let $\varphi: H \to \bZ/m\bZ$ be the non-trivial homomorphism such that $L -\varphi$ is constant (say $a$) in $\bZ/m\bZ$. Write $\varphi(H) = \langle m' \rangle$ with $m'\mid m$ and $1\le m' < m$. For $m'' := m/m' >1$, we have
\begin{align*}
\sum_{h \in H} \sin^2 \frac{\pi\cdot L(h)}{m} &=\frac{1}{2}\sum_{h\in H}\left(1-\cos\left(\frac{2\pi\cdot L(h)}{m}\right)\right)\\
&=\frac{\lt H\rt}{2}-\frac{|H|}{2m''}\sum_{h'\in\bZ/m''\bZ}\cos\left(\frac{2\pi a}{m}+\frac{2\pi h'}{m''}\right)\\
&=\frac{\lt H\rt}{2}-\frac{|H|}{2m''}\mr{Re}\left(\exp\left(\frac{2\pi i a}{m}\right)\sum_{h'\in\bZ/m''\bZ}\exp\left(\frac{2\pi i h'}{m''}\right)\right)\\
&=\frac{\lt H\rt}{2}.
\end{align*}  
The second and third assertions follow from condition \ref{bilinear_cond2} in Section \ref{Sub23} by taking 
$$
L(*)= tB(\mb{g}, \mb{g}' + *) \text{ or } L(*)= tB(\mb{g'}+*, \mb{g})
$$ 
according to which of $B(\mb{g}, \mb{g}' + *)$ or $B(\mb{g'}+*, \mb{g})$ is a non-constant affine map on $H$.
\end{proof}

\subsection{Proof of Theorem \ref{thm: general}: Basic setting} \label{Sub42}

Define
$$
F(n):=\max_{t_{k,l}\in R^\times}\frac{1}{\lt G\rt^n}\sum_{(\mb{g}_1,\cdots,\mb{g}_n)\notin G_{tr}}\prod_{k\in[n]}\lt1-\al_n+\al_n\ze^{U(\mb{g}_k)t_{k,k,n}}\rt\cdot\prod_{(k,l)\in \Omega}\lt1-\al_n+\al_n \ze^{B(\mb{g}_k,\mb{g}_l)t_{k,l,n}}\rt
$$
and let $t_{k,l,n}\in R^\times$ ($(k,l) \in \Omega \cup \Delta$) be elements which achieve its maximum. Then Lemma \ref{lem: extreme point} (with $I = G^{2n} \setminus G_{tr}$ and $J=\Omega \cup \Delta$) implies that
$$
\lt f((x_{k,l,n}))-\frac{\lt G_{tr}\rt}{\lt G\rt^n}\rt=\frac{1}{\lt G\rt^n}\lt\sum_{(\mb{g}_1,\cdots,\mb{g}_n)\notin G_{tr}}\prod_{k\in[n]}\bE\left(\ze^{U(\mb{g}_k)x_{k,k,n}}\right)\cdot\prod_{(k,l)\in \Omega}\bE\left(\ze^{B(\mb{g}_k,\mb{g}_l)x_{k,l,n}}\right)\rt\le F(n).$$
Let $\ga \in (0, \frac{1}{\lt G \rt^2})$ be a sufficiently small real number that will be specified in the proof of Proposition \ref{prop_case2}. For given $\mb{g}_1,\cdots,\mb{g}_n\in G^2$, define
\begin{gather*}
i_{\mb{g}}=i_{\mb{g}}(\mb{g}_1,\cdots,\mb{g}_n):=\#\{k\in[n]:\mb{g}_k=\mb{g}\},\\ W=W_{\mb{g}_1,\cdots,\mb{g}_n}:=\{\mb{g}\in G^2:i_{\mb{g}}\ge \ga n\}.
\end{gather*}
Intuitively, $W$ is the set of elements of $G^2$ that appear frequently among $\mb{g}_1,\cdots,\mb{g}_n$. By the assumption $\ga < \frac{1}{\lt G \rt^2}$, the set $W$ is nonempty.

Now fix $U$ and $B$ satisfying conditions \ref{bilinear_cond1} and \ref{bilinear_cond2} in Section \ref{Sub23}, and let $\ga_1, \ga_2>0$ be real numbers such that $\ga_1$ is sufficiently large and $\ga_2$ is sufficiently small.
For every $n$-tuple $(\mb{g}_1,\cdots,\mb{g}_n) \in (G^2)^n$, exactly one of the following holds:
\begin{enumerate}
\item[$(\mc{C}_1)$]$W$ is not isotropic;
\item[$(\mc{C}_2)$]$W$ is isotropic but $\lt W\rt<\lt G\rt$;
\item[$(\mc{C}_3)$]$W$ is isotropic and $\lt W\rt=\lt G\rt$ (hence it is maximal);
\begin{enumerate}
\item[$(\mc{C}_{3,a}^W)$]$i_\mb{g}>\ga_1\frac{n}{\log n}$ for some $\mb{g}\notin W$;
\item[$(\mc{C}_{3,b}^W)$]$\mc{C}_{3,a}^W$ does not hold, and $\sum_{\mb{g}\notin W}i_\mb{g}>\ga_2\frac{n}{\log n}$;
\item[$(\mc{C}_{3,c}^W)$]$0<\sum_{\mb{g}\notin W}i_\mb{g}\le\ga_2\frac{n}{\log n}$;
\item[$(\mc{C}_{3,d}^W)$]$\sum_{\mb{g}\notin W}i_\mb{g}=0$ so that $i_\mb{g}=0$ for all $\mb{g}\notin W$.
\end{enumerate}
\end{enumerate}
For a condition $\mc{C}$, define 
$$
F_{\mc{C}}(n):=\frac{1}{\lt G\rt^n} \sum_{\substack{(\mb{g}_1,\cdots,\mb{g}_n)\notin G_{tr}\\\mc{C}\text{ is satisfied.}}}\prod_{k\in[n]}\lt1-\al_n+\al_n\ze^{U(\mb{g}_k)t_{k,k,n}}\rt\cdot\prod_{(k,l)\in \Omega}\lt1-\al_n+\al_n\ze^{B(\mb{g}_k,\mb{g}_l)t_{k,l,n}}\rt.
$$
Then $F(n)$ can be written as a finite sum
$$F(n)=F_{\mc{C}_1}(n)+F_{\mc{C}_2}(n)+\sum_{\substack{W\text{ isotropic}\\\lt W\rt=\lt G\rt}}(F_{\mc{C}_{3,a}^W}(n)+F_{\mc{C}_{3,b}^W}(n)+F_{\mc{C}_{3,c}^W}(n)+F_{\mc{C}_{3,d}^W}(n)).$$
Hence to prove Theorem \ref{thm: general}, it suffices to prove that each $F_{\mc{C}}(n)$ converges to 0 as $n\to\infty$.

\subsection{Proof of Theorem \ref{thm: general}: Cases \texorpdfstring{$\mc{C}_1$}{C1} and \texorpdfstring{$\mc{C}_2$}{C2}} \label{Sub43}

\begin{prop}\label{prop_case1}
$\lim_{n\to\infty}F_{\mc{C}_1}(n)=0$.
\end{prop}

\begin{proof}
Suppose that $B(\mb{g},\mb{g}')\ne0$ for some $\mb{g},\mb{g}'\in W$. Then there are at least $\genfrac(){0pt}{}{\lf \ga n\rf}{2}=\Theta(n^2)$ pairs of $(k,l) \in \Omega$ such that $B(\mb{g}_k,\mb{g}_l)\ne0$. By Lemma \ref{lem_ine1},
\begin{align*}
F_{\mc{C}_1}(n)&=\frac{1}{\lt G\rt^{n}} \sum_{\substack{\mb{g}_1,\cdots,\mb{g}_n\in G^2\\\mc{C}_1\text{ is satisfied}}}\prod_{k\in[n]}\lt1-\al_n+\al_n\ze^{t_{k,k,n} U(\mb{g}_k)}\rt\prod_{(k,l)\in \Omega}\lt1-\al_n+\al_n\ze^{t_{k,l,n}  B(\mb{g}_k,\mb{g}_l)}\rt\\&\le\frac{1}{\lt G\rt^n}\sum_{\substack{\mb{g}_1,\cdots,\mb{g}_n\in G^2\\\mc{C}_1\text{ is satisfied}}}\left(1-\frac{\log n}{n}\frac{8c_1}{p^{2d}}\right)^{\Theta(n^2)}\\
&\le \lt G\rt^n\left(1-\frac{\log n}{n}\frac{8c_1}{p^{2d}}\right)^{\Theta(n^2)}\\
&\le\lt G\rt^n\exp\left(-\frac{\log n}{n}\frac{8c_1}{p^{2d}}\cdot \Theta(n^2)\right),\\
&=\lt G\rt^n\exp\left(-\frac{8c_1}{p^{2d}}\cdot\Theta(n\log n)\right),
\end{align*}
and the last term converges to $0$ as $n\to\infty$.
\end{proof}

\begin{prop}\label{prop_case2}
$\lim_{n\to\infty} F_{\mc{C}_2}(n)=0$.
\end{prop}

\begin{proof}
We have
\begin{align*}
F_{\mc{C}_2}(n)&=\frac{1}{\lt G\rt^{n}} \sum_{\substack{\mb{g}_1,\cdots,\mb{g}_n\in G^2\\\mc{C}_2\text{ is satisfied}}}\prod_{k\in[n]}\lt1-\al_n+\al_n\ze^{t_{k,k,n} U(\mb{g}_k)}\rt\prod_{(k,l)\in \Omega}\lt1-\al_n+\al_n\ze^{t_{k,l,n} B(\mb{g}_k,\mb{g}_l)}\rt\\&\le\frac{1}{\lt G\rt^{n}}\sum_{\substack{W_0\text{ isotropic}\\\lt W_0\rt<\lt G\rt}}\sum_{\substack{\mb{g}_1,\cdots,\mb{g}_n\in G^2\\ W_{\mb{g}_1,\cdots,\mb{g}_n}=W_0}}1\\&=\frac{1}{\lt G\rt^{n}}\sum_{\substack{W_0\text{ isotropic}\\0< \lt W_0\rt<\lt G\rt}} \#\left\{(\mb{g}_1,\cdots,\mb{g}_n)\in G^{2n} : W_{\mb{g}_1,\cdots,\mb{g}_n}=W_0\right\}.
\end{align*}
For each nonempty isotropic subset $W_0$ in $G^2$, 
\begin{align*}
&\#\left\{(\mb{g}_1,\cdots,\mb{g}_n)\in G^{2n} :W_{\mb{g}_1,\cdots,\mb{g}_n}=W_0\right\} \\
=\,&\sum_{\substack{i_\mb{g}\ge\ga n\text{ for all }\mb{g}\in W_0\\i_\mb{g}<\ga n\text{ for all }\mb{g}\notin W_0 \\ \sum_{\mb{g} \in G^2} i_\mb{g}=n}}\frac{n!}{\prod_{\mb{g}\in G^2}i_{\mb{g}}!}\\
=\,&\sum_{\substack{i_\mb{g}<\ga n\text{ for all }\mb{g}\notin W_0}}\frac{(\sum_{\mb{g}\notin W_0}i_\mb{g})!}{\prod_{\mb{g}\notin W_0}i_\mb{g}!}\left(\sum_{\substack{i_\mb{g}\ge\ga n\text{ for all }\mb{g}\in W_0 \\ \sum_{\mb{g} \in W_0} i_\mb{g}=n-\sum_{\mb{g} \notin W_0}i_\mb{g}}}\frac{n!}{(\sum_{\mb{g}\notin W_0}i_\mb{g})! \cdot\prod_{\mb{g}\in W_0}i_\mb{g}!}\right)\\
\le\,&\sum_{i<\ga(\lt G\rt^2-\lt W_0\rt)n} \left(\sum_{\substack{i_\mb{g}\ge 0 \text{ for all }\mb{g}\notin W_0 \\ \sum_{\mb{g} \notin W_0} i_\mb{g} = i}}\frac{i!}{\prod_{\mb{g}\notin W_0}i_\mb{g}!} \right) \left(\sum_{\substack{i_\mb{g}\ge\ga n\text{ for all }\mb{g}\in W_0\\ \sum_{\mb{g} \in W_0} i_\mb{g} = n-i}}\frac{n!}{i!\cdot\prod_{\mb{g}\in W_0}i_\mb{g}!}\right)\\
=\,&\sum_{\substack{i<\ga(\lt G\rt^2-\lt W_0\rt)n}}(\lt G\rt^2-\lt W_0\rt)^{i}\genfrac(){0pt}{}{n}{i}\left(\sum_{\substack{i_\mb{g}\ge\ga n\text{ for all }\mb{g}\in W_0\\ \sum_{\mb{g} \in W_0} i_\mb{g} = n-i}}\frac{(n-i)!}{\prod_{\mb{g}\in W_0}i_\mb{g}!}\right)\\
\le\,&\sum_{\substack{i<\ga\lt G\rt^2n}}\lt G\rt^{2i}\genfrac(){0pt}{}{n}{i}\lt W_0\rt^{n-i}.
\end{align*}
Set $\ga_0=\ga\lt G\rt^2$, then
$$
\sum_{\substack{i<\ga_0n}}\lt G\rt^{2i}\genfrac(){0pt}{}{n}{i}\lt W_0\rt^{n-i} \le\ga_0n\lt G\rt^{2\ga_0n}\genfrac(){0pt}{}{n}{\lfloor \ga_0n\rfloor}\lt W_0\rt^{n}.
$$
By \cite[Example 12.1.3]{CT06}, 
$$
\binom{n}{\lf \ga_0 n \rf} \le 2^{nH(\ga_0)} < \frac{1}{(\ga_0^{\ga_0}(1-\ga_0)^{1-\ga_0})^n},
$$
where $H(q) = - q \log q - (1-q) \log (1-q)$ is the binary entropy function.
It follows that
\begin{equation} \label{eq_case2}
\begin{split}
F_{\mc{C}_2}(n) &\le\frac{1}{\lt G\rt^n}\sum_{\substack{W_0\text{ isotropic}\\\lt W_0\rt<\lt G\rt}}\ga_0n\left(\frac{\lt G\rt^{2\ga_0}}{\ga_0^{\ga_0}(1-\ga_0)^{1-\ga_0}}\right)^n\lt W_0\rt^n\\&\le N_G\cdot\ga_0n\left(\frac{\lt G\rt-1}{\lt G\rt}\right)^n\left(\frac{\lt G\rt^{2\ga_0}}{\ga_0^{\ga_0}(1-\ga_0)^{1-\ga_0}}\right)^n,
\end{split}
\end{equation}
where $N_G$ is the number of isotropic subsets $W_0 \subseteq G^2$ with $\lt W_0 \rt < \lt G \rt$.
Since 
$$
\lim_{\ga_0 \to 0^+} \frac{\lt G\rt^{2\ga_0}}{\ga_0^{\ga_0}(1-\ga_0)^{1-\ga_0}} = 1,
$$
we may choose $\ga>0$ sufficiently small so that $\ga_0=\ga\lt G\rt^2$ satisfies
$$
\frac{\lt G \rt-1}{\lt G\rt} \cdot \frac{\lt G\rt^{2\ga_0}}{\ga_0^{\ga_0}(1-\ga_0)^{1-\ga_0}} < 1.
$$
In this case, the last term of \eqref{eq_case2} converges to $0$ as $n\to\infty$.
\end{proof}

\subsection{Proof of Theorem \ref{thm: general}: Case \texorpdfstring{$\mc{C}_3$}{C3}} \label{Sub44}

\begin{prop}\label{prop_case3a}
$\lim_{n\to\infty}F_{\mc{C}^W_{3,a}}=0$ for every isotropic $W \subseteq G^2$ with $\lt W\rt=\lt G\rt$.
\end{prop}

\begin{proof}
Suppose that $i_\mb{g}>\ga_1\frac{n}{\log n}$ for some $\mb{g}\notin W$. Then there exists $\mb{g}'\in W$ such that either $B(\mb{g},\mb{g}')\ne0$ or $B(\mb{g}',\mb{g})\ne0$. Without loss of generality, assume that  $B(\mb{g},\mb{g}')\ne0$.\footnote{Except for the non-symmetric case, this implies that $B(\mb{g}',\mb{g}) \ne 0$.} 
By Lemma \ref{lem_ine1},
\begin{align*}
F_{\mc{C}_{3,a}^W}(n)&=\frac{1}{\lt G\rt^{n}} \sum_{\substack{\mb{g}_1,\cdots,\mb{g}_n\in G^2\\\mc{C}_{3,a}^W\text{ is satisfied}}}\prod_{k\in[n]}\lt1-\al_n+\al_n\ze^{t_{k,k,n} U(\mb{g}_k)}\rt\prod_{(k,l)\in \Omega}\lt1-\al_n+\al_n\ze^{t_{k,l,n} B(\mb{g}_k,\mb{g}_l)}\rt\\
&\le\frac{1}{\lt G\rt^n}\sum_{\substack{\mb{g}_1,\cdots,\mb{g}_n\in G^2\\\mc{C}_{3,a}^W\text{ is satisfied}}}\left(1-\frac{\log n}{n}\frac{8c_1}{p^{2d}}\right)^{\ga_1\frac{n}{\log n} \cdot \ga n}\\
&=\lt G\rt^n\left(1-\frac{\log n}{n}\frac{8c_1}{p^{2d}}\right)^{\ga\ga_1\frac{n^2}{\log n}}\\
&\le\lt G\rt^n\exp\left(-\frac{\log n}{n}\frac{8c_1}{p^{2d}} \cdot\ga\ga_1\frac{n^2}{\log n}\right)\\
&=\lt G\rt^n\exp\left(-\frac{8c_1\ga\ga_1}{p^{2d}} n\right).
\end{align*}
Now choose a sufficiently large $\ga_1$ (depending on $\ga$) so that the last term converges to $0$ as $n\to\infty$.
\end{proof}

\begin{prop}\label{prop_case3b}
$\lim_{n\to\infty}F_{\mc{C}^W_{3,b}}=0$ for every isotropic $W \subseteq G^2$ with $\lt W\rt=\lt G\rt$.
\end{prop}

\begin{proof}
Suppose that $i_\mb{g}\le\ga_1\frac{n}{\log n}$ for all $\mb{g}\notin W$ and $\sum_{\mb{g}\notin W}i_\mb{g}>\ga_2\frac{n}{\log n}$, so $i_\mb{g}>\frac{\ga_2}{\lt G\rt^2}\frac{n}{\log n}$ for some $\mb{g}\notin W$. As before, we may assume that there exists $\mb{g}'\in W$ such that  $B(\mb{g},\mb{g}')\ne 0$. By Lemma \ref{lem_ine1},
\begin{align*}
F_{\mc{C}_{3,b}^W}(n)&=\frac{1}{\lt G\rt^{n}} \sum_{\substack{\mb{g}_1,\cdots,\mb{g}_n\in G^2\\\mc{C}_{3,b}^W\text{ is satisfied}}}\prod_{k\in[n]}\lt1-\al_n+\al_n\ze^{t_{k,k,n} U(\mb{g}_k)}\rt\prod_{(k,l)\in \Omega}\lt1-\al_n+\al_n\ze^{t_{k,l,n} B(\mb{g}_k,\mb{g}_l)}\rt\\
&\le\frac{1}{\lt G\rt^{n}} \sum_{\substack{\mb{g}_1,\cdots,\mb{g}_n\in G^2\\\mc{C}_{3,b}^W\text{ is satisfied}}}\exp\left(-\frac{\log n}{n} \frac{8c_1}{p^{2d}} \cdot \ga n \cdot \frac{\ga_2}{\lt G\rt^2}\frac{n}{\log n} \right)\\
&=\frac{1}{\lt G\rt^n}\cdot\frac{1}{K^n}\cdot\#\left\{(\mb{g}_1,\cdots,\mb{g}_n)\in G^{2n}:\mc{C}_{3,b}^W\text{ is satisfied}\right\}\\&\le\frac{1}{K^n\lt G\rt^n}\cdot\#\left\{(\mb{g}_1,\cdots,\mb{g}_n)\in G^{2n}:i_\mb{g}\le\ga_1\frac{n}{\log n}\text{ for all }\mb{g}\notin W\right\}
\end{align*}
where $K:=\exp\left(\frac{8c_1\ga\ga_2}{\lt G\rt^2p^{2d}}\right)>1$. 
Write $G^2\setminus W=\{\mb{h}_1,\cdots,\mb{h}_{\lt G\rt^2-\lt G\rt}\}$. Then for all sufficiently large $n$,
\begin{align*}
&\frac{1}{K^n\lt G\rt^n}\cdot\#\left\{(\mb{g}_1,\cdots,\mb{g}_n)\in G^{2n}:i_\mb{g}\le\ga_1\frac{n}{\log n}\text{ for all }\mb{g}\notin W\right\}\\
= \, &\frac{1}{K^n\lt G\rt^n}\cdot\sum_{k_1,\cdots,k_{\lt G\rt^2-\lt G\rt}\le\ga_1\frac{n}{\log n}}\#\left\{(\mb{g}_1,\cdots,\mb{g}_n)\in G^{2n}:i_{\mb{h}_i}=k_i\text{ for all }i\right\}\\
= \, &\frac{1}{K^n\lt G\rt^n}\cdot\sum_{k_1,\cdots,k_{\lt G\rt^2-\lt G\rt}\le\ga_1\frac{n}{\log n}}\frac{n!}{\prod_ik_i!\cdot(n-\sum_ik_i)!}\lt G\rt^{n-\sum_i k_i}\\
\le \, & \frac{1}{K^n}\cdot\sum_{k_1,\cdots,k_{\lt G\rt^2-\lt G\rt}\le\ga_1\frac{n}{\log n}}\frac{n^{\sum_ik_i}}{\prod_ik_i!}\\
= \, & \frac{1}{K^n}\cdot\prod_{i=1}^{\lt G\rt^2-\lt G\rt} \left(\sum_{k_i\le\ga_1\frac{n}{\log n}}\frac{n^{k_i}}{k_i!}\right)\\
\le \, &\frac{1}{K^n} \cdot \left(\frac{2n^{\left\lfloor\ga_1\frac{n}{\log n}\right\rfloor}}{\left\lfloor\ga_1\frac{n}{\log n}\right\rfloor!}\right)^{\lt G\rt^2-\lt G\rt}\\
\le \, &\frac{2^{\lt G\rt^2-\lt G\rt}}{K^n}\cdot\left(\frac{en}{\left\lfloor\ga_1\frac{n}{\log n}\right\rfloor}\right)^{(\lt G\rt^2-\lt G\rt)\left\lfloor\ga_1\frac{n}{\log n}\right\rfloor}\\
\le \, & \frac{2^{\lt G\rt^2-\lt G\rt}}{K^n}\cdot\left(\frac{3 \log n}{\ga_1}\right)^{(\lt G\rt^2-\lt G\rt)\ga_1\frac{n}{\log n}}.
\end{align*}
The second inequality holds since $\ga_1\frac{n}{\log n} < \frac{n}{2}$ for all sufficiently large $n$, and the third inequality follows from the estimate $n!\ge (\frac{n}{e})^n$. Now the last term converges to $0$ as $n \to \infty$, which completes the proof.
\end{proof}

Let $\kappa \in \{ 1, 2 \}$ be a constant such that $\Omega=\Omega_{\kappa}$ in the notation of Section \ref{Sub23}. Recall that when $\kappa=2$, the pairing $B$ is either symmetric or alternating by condition \ref{bilinear_cond3} in Section \ref{Sub23}.

\begin{prop}\label{prop_case3c}
$\lim_{n\to\infty} F_{\mc{C}^W_{3,c}}=0$ for every isotropic $W \subseteq G^2$ with $\lt W\rt=\lt G\rt$.
\end{prop}

\begin{proof}
Suppose that $0<\sum_{\mb{g}\notin W}i_\mb{g}\le\ga_2\frac{n}{\log n}$. When $\kappa=2$, for each $(k,l) \in [n] \times [n] \setminus (\Omega_{\kappa} \cup \Delta)$, define
$$
t_{k,l,n} := \left\{\begin{matrix}
t_{l,k,n} & \text{if } B \text{ is symmetric},\\
-t_{l,k,n} & \text{if } B \text{ is alternating}.
\end{matrix}\right.
$$
Then
\begin{align*}
F_{\mc{C}_{3,c}^W}(n)&=\frac{1}{\lt G\rt^{n}} \sum_{\substack{\mb{g}_1,\cdots,\mb{g}_n\in G^2\\\mc{C}_{3,c}^W\text{ is satisfied}}}\prod_{k\in[n]}\lt1-\al_n+\al_n\ze^{t_{k,k,n} U(\mb{g}_k)}\rt\prod_{(k,l)\in \Omega_{\kappa}}\lt1-\al_n+\al_n\ze^{t_{k,l,n} B(\mb{g}_k,\mb{g}_l)}\rt\\
&=\frac{1}{\lt G\rt^{n}} \sum_{\substack{X\subseteq[n]\\0<\lt X\rt\le\ga_2\frac{n}{\log n}}}\sum_{\substack{\mb{g}_1,\cdots,\mb{g}_n\in G^2\\\{k\in[n]:\mb{g}_k\notin W\}=X\\i_\mb{g}\ge\ga n\text{ for all }\mb{g}\in W}}\prod_{k\in[n]}\lt1-\al_n +\al_n\ze^{t_{k,k,n} U(\mb{g}_k)} \rt \prod_{(k,l)\in \Omega_{\kappa}}\lt 1-\al_n+\al_n\ze^{t_{k,l,n} B(\mb{g}_k,\mb{g}_l)}\rt\\
&\le\frac{1}{\lt G\rt^{n}} \sum_{\substack{X\subseteq[n]\\0<\lt X\rt\le\ga_2\frac{n}{\log n}}} \sum_{\substack{\mb{g}_1,\cdots,\mb{g}_n\in G^2\\\{k\in[n]:\mb{g}_k\notin W\}=X}}\prod_{\substack{x\in X\\y\notin X}}\lt1-\al_n+\al_n\ze^{t_{x,y,n} B(\mb{g}_x,\mb{g}_y)}\rt^{1/\kappa}\lt1-\al_n+\al_n\ze^{t_{y,x,n} B(\mb{g}_y,\mb{g}_x)}\rt^{1/\kappa}.
\end{align*}
For each $y \in [n] \setminus X$, Lemma \ref{lem_ine1} implies that
\begin{align*}
&\prod_{x\in X}\lt1-\al_n+\al_n\ze^{t_{x,y,n} B(\mb{g}_x,\mb{g}_y)}\rt^{1/\kappa}\lt1-\al_n+\al_n\ze^{t_{y,x,n} B(\mb{g}_y,\mb{g}_x)}\rt^{1/\kappa}\\
\le \, & \prod_{x\in X}\left(1-\frac{2c_1\log n}{n}\sin^2\frac{\pi t_{x,y,n}B(\mb{g}_x,\mb{g}_y)}{p^d}\right)^{1/\kappa}\left(1-\frac{2c_1\log n}{n}\sin^2\frac{\pi t_{y,x,n}B(\mb{g}_y,\mb{g}_x)}{p^d}\right)^{1/\kappa}\\
\le \, & \exp\left(-\frac{2c_1}{\kappa}\frac{\log n}{n}S_y\right)\\
\le \, & 1-\frac{2c_1}{\kappa}\frac{\log n}{n}S_y+\frac{1}{2}\left(\frac{2c_1}{\kappa}\frac{\log n}{n}S_y\right)^2\\
\le \, & 1-\frac{2c_1}{\kappa}\frac{\log n}{n}\left(1-\frac{2c_1}{\kappa}\frac{\log n}{n} \lt X\rt\right)S_y\\
\le \, & 1-\frac{2c_1}{\kappa}\frac{\log n}{n}\left(1-\frac{2c_1}{\kappa}\frac{\log n}{n} \cdot \ga_2\frac{n}{\log n}\right)S_y\\
\le \, & 1-\frac{2c_2}{\kappa} \frac{\log n}{n}S_y,
\end{align*}
where 
$$
S_y:=\sum_{x\in X}\left(\sin^2\frac{\pi t_{x,y,n}B(\mb{g}_x,\mb{g}_y)}{p^d}+\sin^2\frac{\pi t_{y,x,n}B(\mb{g}_y,\mb{g}_x)}{p^d}\right)\le2\lt X\rt
$$
and $\ga_2>0$ is taken to satisfy $c_2 = c_1(1-2c_1 \ga_2)$. Consequently, for all sufficiently large $n$,
\begin{align*}
&\frac{1}{\lt G\rt^{n}} \sum_{\substack{X\subseteq[n]\\0<\lt X\rt\le\ga_2\frac{n}{\log n}}} \sum_{\substack{\mb{g}_1,\cdots,\mb{g}_n\in G^2\\ \{k\in[n]:\mb{g}_k\notin W\}=X}}\prod_{\substack{x\in X\\y\notin X}}\lt1-\al_n+\al_n\ze^{t_{x,y,n} B(\mb{g}_x,\mb{g}_y)}\rt^{1/\kappa}\lt1-\al_n+\al_n\ze^{t_{y,x,n} B(\mb{g}_y,\mb{g}_x)}\rt^{1/\kappa}\\
\le\,& \frac{1}{\lt G\rt^{n}} \sum_{\substack{X\subseteq[n]\\0<\lt X\rt\le\ga_2\frac{n}{\log n}}} \sum_{\substack{(\mb{g}_x)\in(G^2\setminus W)^X\\(\mb{g}_y)\in W^{G^2\setminus X}}}\prod_{\substack{y\notin X}}\left(1-\frac{2c_2}{\kappa}\frac{\log n}{n}S_y\right)\\
=\,& \frac{1}{\lt G\rt^{n}}\sum_{\substack{X\subseteq[n]\\0<\lt X\rt\le\ga_2\frac{n}{\log n}}} \sum_{\substack{(\mb{g}_x) \in(G^2\setminus W)^X}}\prod_{\substack{y\notin X}}\sum_{\mb{g}_y\in W}\left(1-\frac{2c_2}{\kappa}\frac{\log n}{n}S_y\right)\\
=\,& \frac{1}{\lt G\rt^{n}}\sum_{\substack{X\subseteq[n]\\0<\lt X\rt\le\ga_2\frac{n}{\log n}}}\sum_{\substack{(\mb{g}_x)\in(G^2\setminus W)^X}}\prod_{\substack{y\notin X}}\left(\lt G\rt-\frac{2c_2}{\kappa}\frac{\log n}{n}\sum_{\mb{g}_y\in W}S_y\right)\\
\le\,& \frac{1}{\lt G\rt^{n}}\sum_{\substack{X\subseteq[n]\\0<\lt X\rt\le\ga_2\frac{n}{\log n}}}\sum_{\substack{(\mb{g}_x)\in(G^2\setminus W)^X}}\prod_{\substack{y\notin X}}\left(\lt G\rt-c_2\frac{\log n}{n}\cdot\lt X\rt\lt G\rt\right)\\
=\,& \frac{1}{\lt G\rt^n}\sum_{\substack{X\subseteq[n]\\0<\lt X\rt\le\ga_2\frac{n}{\log n}}}(\lt G\rt^2-\lt G\rt)^{\lt X\rt}\cdot\lt G\rt^{n-\lt X\rt}\left(1-c_2\frac{\log n}{n}\lt X\rt\right)^{n-\lt X\rt}\\
=\,& \sum_{\substack{X\subseteq[n]\\0<\lt X\rt\le\ga_2\frac{n}{\log n}}}(\lt G\rt-1)^{\lt X\rt}\cdot\left(1-c_2\frac{\log n}{n}\lt X\rt\right)^{n-\lt X\rt}\\
\le\,& \sum_{\substack{X\subseteq[n]\\0<\lt X\rt\le\ga_2\frac{n}{\log n}}}(\lt G\rt-1)^{\lt X\rt}\cdot\exp\left(-c_2\frac{\log n}{n}\cdot\lt X\rt(n-\lt X\rt)\right)\\
\le\,& \sum_{\substack{X\subseteq[n]\\0<\lt X\rt\le\ga_2\frac{n}{\log n}}}(\lt G\rt-1)^{\lt X\rt}\cdot\exp\left(-c_3 \lt X\rt\cdot\log n\right) \\
=\,& \sum_{\substack{X\subseteq[n]\\0<\lt X\rt\le\ga_2\frac{n}{\log n}}}\left(\frac{\lt G\rt-1}{n^{c_3}}\right)^{\lt X\rt}\\
\le\,& \left( 1 + \frac{\lt G\rt-1}{n^{c_3}} \right)^n - 1.
\end{align*}
Here, the second inequality (which is an equality when $\kappa =2$) follows from Lemma \ref{lem: sin sum}, and the fourth inequality holds since $c_2 \left(1 - \frac{\ga_2}{\log n} \right) > c_3$ for all sufficiently large $n$. Now the last term converges to $0$ as $n\to\infty$.
\end{proof}

\begin{prop}\label{prop_case3d}
$\lim_{n\to\infty}F_{\mc{C}^W_{3,d}}=0$ for every isotropic $W \subseteq G^2$ with $\lt W\rt=\lt G\rt$.
\end{prop}

\begin{proof}
Suppose that $i_\mb{g}=0$ for all $\mb{g}\notin W$. Then $U(\mb{g}_k)\ne0$ for some $\mb{g}_k\in W$, since $(\mb{g}_1,\cdots,\mb{g}_n)\notin G_{tr}$. Hence, applying Lemma \ref{lem_ine1}, we have
\begin{align*}
F_{\mc{C}_{3,d}^W}(n) &=\frac{1}{\lt G\rt^{n}} \sum_{\substack{\mb{g}_1,\cdots,\mb{g}_n\in G^2\\\mc{C}_{3,d}^W\text{ is satisfied}}}\prod_{k\in[n]}\lt1-\al_n+\al_n\ze^{t_{k,k,n} U(\mb{g}_k)}\rt\prod_{(k,l)\in \Omega}\lt1-\al_n+\al_n\ze^{t_{k,l,n} B(\mb{g}_k,\mb{g}_l)}\rt\\
&\le\frac{1}{\lt G\rt^n}\sum_{(\mb{g}_1,\cdots,\mb{g}_n)\in W^n\setminus G_{tr}}\left(1-\frac{\log n}{n}\frac{8c_1}{p^{2d}}\right)^{\ga n}\\
& \le\exp\left(-\frac{\log n}{n}\frac{8c_1}{p^{2d}}\cdot\ga n\right)\\
&=n^{-\frac{8c_1\ga}{p^{2d}}},
\end{align*}
where the last term converges to $0$ as $n\to\infty$.
\end{proof}

\section{Sharpness of the \texorpdfstring{$\al_n$}{aln} threshold} \label{Sec5}

In this section, we prove the sharpness of the $\frac{\log n}{n}$ threshold in our main results.

\begin{prop} \label{prop_optimal}
Suppose that $A(n)$ is a $\frac{\log n}{n}$-balanced random matrix in $\mf{M}_n$ over $\fp$ (or $\zp$) whose entries are i.i.d. copies of a random element $z$ satisfying $\bP(z = 0) = 1 - \frac{\log n}{n}$ and $\bP(z = 1) = \frac{\log n}{n}$. (If $\mf{M}_n=\mr{Sym}_n$ (resp. $\mf{M}_n=\mr{Alt}_n$), then the upper triangular (resp. strictly upper triangular) entries are i.i.d. copies of $z$.) Then for every $k \ge 1$,
\begin{equation} \label{eq_optimal}
\liminf_{n \to \infty} \bP(A(n) \text{ has at least } k \text{ zero columns}) \ge \frac{1}{(k+1)!}.
\end{equation}
\end{prop}

\begin{proof}
Let $Y_i$ denote the event that the $i$-th column of $A(n)$ is zero, and set $I_i = \mathbf{1}_{Y_i}$ and $N=\sum_{i=1}^{n} I_i$. For all nonnegative integers $m$ and $k$, we have
$$
\binom{m}{k} - k \binom{m}{k+1} \le \mathbf{1}_{m \ge k},
$$
where $\mathbf{1}_{m \ge k}$ denotes the indicator function of the event $m \ge k$, i.e. $\mathbf{1}_{m \ge k}=1$ if $m \ge k$ and $\mathbf{1}_{m \ge k}=0$ otherwise. 
Therefore
$$
\bP(A(n) \text{ has at least } k \text{ zero columns})
= \bE(\mathbf{1}_{N \ge k})
\ge \bE \binom{N}{k} - k \bE \binom{N}{k+1}.
$$
By the definition of $N$, 
$$
\bE \binom{N}{r} = \sum_{1 \le i_1 < \cdots < i_r \le n} \bE(I_{i_1} \cdots I_{i_r})
= \sum_{1 \le i_1 < \cdots < i_r \le n} \bP(Y_{i_1} \cap \cdots \cap Y_{i_r})
$$
for every $r \ge 1$. For $u_n = \frac{\log n}{n}$, 
$$
\bP(Y_{i_1} \cap \cdots \cap Y_{i_r}) = \left\{\begin{matrix}
(1-u_n)^{rn} & \text{if } \mf{M}_n = \M_n, \\
(1-u_n)^{rn - \frac{r(r-1)}{2}} & \text{if } \mf{M}_n = \mr{Sym}_n, \\
(1-u_n)^{rn- \frac{r(r+1)}{2}} & \text{if } \mf{M}_n = \mr{Alt}_n
\end{matrix}\right.
$$
so
\begin{align*}
\bE \binom{N}{k} - k \bE \binom{N}{k+1}
& \ge \binom{n}{k} (1-u_n)^{kn} - k \binom{n}{k+1} (1-u_n)^{(k+1)n - \frac{(k+1)(k+2)}{2}} \\
& = \binom{n}{k} (1-u_n)^{kn} \left( 1 - k \frac{n-k}{k+1} (1-u_n)^{n-\frac{(k+1)(k+2)}{2}} \right).
\end{align*}
Since
$$
n(1-u_n)^n 
= n^{1+\frac{\log(1-u_n)}{u_n}}
= n^{-\frac{u_n}{2}+O(u_n^2)}
= e^{-\frac{(\log n)^2}{2n} + O\left( \frac{(\log n)^3}{n^2} \right)},
$$
we have $n(1-u_n)^n \to 1$ as $n \to \infty$. Hence, for every $k\ge 1$,
\begin{equation*}
\lim_{n \to \infty} \binom{n}{k} (1-u_n)^{kn} \left( 1 - k \frac{n-k}{k+1} (1-u_n)^{n-\frac{(k+1)(k+2)}{2}} \right) = \frac{1}{k!} \left( 1 - \frac{k}{k+1} \right) = \frac{1}{(k+1)!}. \qedhere
\end{equation*}
\end{proof}

Now let $A(n)$ be a $\frac{\log n}{n}$-balanced random matrix in $\mf{M}_n$ over $\fp$ defined as in Proposition \ref{prop_optimal}, and let $B(n)$ be a uniform random matrix in $\mf{M}_n$ over $\fp$. 
Then for all sufficiently large $n$,
$$
\bP(\rank(A(n)) \le n-k)
\ge \bP(A(n) \text{ has at least } k \text{ zero columns}) \\
\ge \frac{1}{2(k+1)!}.
$$
However, if $\mf{M}_n$ is one of $\M_n$, $\mr{Sym}_n$, or $\mr{Alt}_n$, then Theorem \ref{thm_nonsym_Fp} and Corollary \ref{main thm_fp} imply
$$
\bP(\rank(B(n)) \le n-k) = O \left( p^{-\frac{k(k-1)}{2}} \right).
$$
Since $\frac{1}{2(k+1)!}$ dominates $p^{-\frac{k(k-1)}{2}}$ for all sufficiently large $k$, the lower bound for $A(n)$ is asymptotically much larger than the corresponding probability for the uniform matrix model $B(n)$. 
Hence the distribution of $\cok(A(n))$ cannot converge to the limiting distribution of $\cok(B(n))$. This shows that universality fails when $\al_n = \frac{\log n}{n}$. In particular, Theorem \ref{thm_nonsym_Fp} and Corollary \ref{main thm_fp} cannot be extended to the critical case $c=1$. Moreover, if Theorem \ref{main thm} were valid for $c=1$, then reducing modulo $p$ would imply Theorem \ref{thm_nonsym_Fp} and Corollary \ref{main thm_fp} for $c=1$, which is impossible. Hence Theorem \ref{main thm} also fails in the critical case $c=1$.

\section{Extending to a finite set of primes} \label{Sec6}

In this section, we briefly sketch how to extend our main results to the $\mc{P}$-primary part of the cokernel of a random integral matrix, where $\mc{P}$ is a finite set of primes. 
For a finite set of primes $\mc{P}$, a \emph{$\mc{P}$-group} is a finite abelian group whose order is a product of powers of primes in $\mc{P}$. Let $S_{\mc{P}}$ denote the set of all finite abelian groups of the form $G \times G$ for some $\mc{P}$-group $G$. For $H \in S_{\mc{P}}$, $\mr{Sp}(H)$ is defined as in the introduction.
For a finitely generated abelian group $G$, write $G_{\mc{P}} := \bigoplus_{p \in \mc{P}} (G \otimes \zp)$.
A random matrix in $\mf{M}_n(\bZ)$ is $\al$-\emph{balanced} if its reduction modulo $p$ is $\al$-balanced as a random matrix in $\mf{M}_n(\fp)$ for every prime $p$. 

\begin{thm} \label{main thm finite prime case}
Let $c>1$ be a constant and $\al_n = \frac{c \log n}{n}$. Suppose that $A(n)$ is an $\al_n$-balanced random matrix in $\mf{M}_n(\bZ)$ for each $n \ge 1$, where $\mf{M}_n$ is one of $\M_n$, $\mr{Sym}_n$ or $\mr{Alt}_n$. Let $H$ be a finite abelian group and let $\mc{P}$ be a finite set of primes containing all prime divisors of $|H|$.
\begin{enumerate}[label=(\alph*)]
    \item (Non-symmetric case) If $\mf{M}_n = \M_n$, then
    \begin{equation*}
    \lim_{n \to \infty} \bP\left(\cok(A(n))_{\mc{P}} \cong H\right) 
    = \frac{1}{|\Aut(H)|} \prod_{p \in \mc{P}}\prod_{i=1}^{\infty} (1-p^{-i}).
    \end{equation*}
    
    \item (Symmetric case) If $\mf{M}_n = \mr{Sym}_n$, then
    \begin{equation*}
    \lim_{n \to \infty} \bP\left(\cok(A(n))_{\mc{P}} \cong H\right) 
    = \frac{\# \{ \text{symmetric, bilinear, perfect } \phi : H \times H \to \bC^* \}}{|H||\Aut(H)|} \prod_{p \in \mc{P}}\prod_{i=1}^{\infty} (1-p^{1-2i}).
    \end{equation*}

    \item (Alternating case) If $\mf{M}_n = \mr{Alt}_n$, then
    \begin{equation*}
    \lim_{n \to \infty} \bP\left(\cok(A(2n))_{\mc{P}} \cong H\right) 
    = \left\{\begin{matrix}
    \frac{|H|}{|\mr{Sp}(H)|} \prod_{p\in \mc{P}}\prod_{i=1}^{\infty} (1-p^{1-2i}) & \text{if } H \in S_{\mc{P}},\\
    0 & \text{otherwise}
    \end{matrix}\right.
    \end{equation*}
    and 
    \begin{equation*}
    \lim_{n \to \infty} \bP\left(\cok(A(2n+1))_{\mc{P}} \cong \bigoplus_{p \in \mc{P}} \zp \times H \right) 
    = \left\{\begin{matrix}
    \frac{1}{|\mr{Sp}(H)|} \prod_{p \in \mc{P}}\prod_{i=2}^{\infty} (1-p^{1-2i}) & \text{if } H \in S_{\mc{P}},\\
    0 & \text{otherwise}.
    \end{matrix}\right.
    \end{equation*}
\end{enumerate}
\end{thm}

Our proof of Theorem \ref{main thm finite prime case} is based on the computation of the moments. 

\begin{thm} \label{moment thm finite prime case}
Let $A(n)$ be as in Theorem \ref{main thm finite prime case}. Then for any finite abelian group $G$, 
$$
\lim_{n \to \infty} \bE(\# \Sur(\cok(A(n)), G)) = \left\{\begin{matrix}
1 & \text{if} \quad \mf{M}_n = \M_n, \\
\lt \wedge^2 G \rt & \text{if} \quad \mf{M}_n = \mr{Sym}_n, \\
\lt \mr{Sym}^2 G \rt & \text{if} \quad \mf{M}_n = \mr{Alt}_n.
\end{matrix}\right.
$$
\end{thm}

For $|H| = \prod_{p \in \mc{P}} p^{e_p}$ ($e_p \ge 0$), let $m := \prod_{p \in \mc{P}} p^{e_p+1}$ and let $R=\bZ/m\bZ$. 
Let $\ol{A}(n)$ be the reduction of $A(n)$ modulo $m$, which is a $\al_n$-balanced random matrix in $\mf{M}_n(R)$. Then 
$$
\cok(\ol{A}(n)) \cong \cok(A(n)) \otimes \bZ/m\bZ
$$ 
and
$$
\cok(\ol{A}(n)) \cong H \quad \text{if and only if} \quad \cok(A(n))_{\mc{P}} \cong H.
$$
Hence Theorem \ref{main thm finite prime case}(a) and (b) follow from \cite[Theorem 3.1]{Woo19} together with Theorem \ref{moment thm finite prime case}. For part (c), we refer to the proof of \cite[Theorem 1.13]{NW25}. The proof uses \cite[Theorem 3.1]{NW25}, which corresponds to Theorem \ref{moment thm finite prime case} in the case $\mf{M}_n = \mr{Alt}_n$. We remark that in \cite[Theorem 1.13]{NW25}, the condition
$$
\lim_{n \to \infty} \bP(\rank(A(2n+1)) = 2n) = 1
$$
implies that
$$
\lim_{n \to \infty} \bP\left(\cok_{\mr{tors}}(A(2n+1))_{\mc{P}} \cong H \right) =
\lim_{n \to \infty} \bP\left(\cok(A(2n+1))_{\mc{P}} \cong \bigoplus_{p \in \mc{P}} \zp \times H \right).
$$

Fix a finite abelian group $G$ and let $m$ be a positive integer such that $mG=0$. Let $\ze := \exp\left(\frac{2\pi i}{m}\right)$ and let
$$
G = \bZ/m_1\bZ \times \cdots \times \bZ/m_r\bZ,
$$
where $m_r \mid m_{r-1} \mid \cdots\mid m_1 \mid m$ with $m_i \ge 2$. Define a perfect $R$-bilinear pairing $\cdot: G \times G \to R$ by 
$$
(x_1 + m_1\bZ, \ldots, x_r+m_r\bZ)\cdot (y_1 + m_1\bZ, \ldots, y_r+m_r\bZ):= \sum_{i=1}^r \frac{m}{m_i} x_iy_i + m\bZ.
$$
Following the computation of Section \ref{Sub22}, we obtain
\begin{align*}
\bE(\# \Hom(\cok(A(n)), G)) 
& = \sum_{F \in \Hom(\bZ^n, G)} \bP(F(A(n))=0) \\
& = \sum_{F \in \Hom(R^n, G)} \bP(F(\ol{A}(n))=0) \\
& = \frac{1}{\lt G\rt^n} \sum_{\substack{g_1,\cdots,g_n\in G\\ h_1,\cdots,h_n\in G}}\bE\left(\ze^{\sum_{k,l\in[n]}\ol{A}(n)_{kl}(g_k \cdot h_l)}\right).
\end{align*}

The proof follows the same general strategy as in the case of random matrices over $\zp$. The only new feature is that, over 
$R=\bZ/m\bZ$, the relevant extreme points are no longer supported on just two values as in Lemma \ref{lem: extreme point}. The next lemma characterizes these points and provides the analogue of Lemma \ref{lem: extreme point} in the present setting.

\begin{lem}\label{lem: extreme point2}
Let $I$ and $J$ be finite sets, $b_{i,j} \in R$ for every $i \in I$, $j \in J$ and $\al \in (0,\frac{1}{m})$. Then for every $(x_j)_{j\in J}\in\mr{Ran}(R,\al)^J$, we have 
$$
\lt\sum_{i\in I}\prod_{j\in J}\bE\left(\ze^{b_{i,j}x_j}\right)\rt
\le\sum_{i\in I}\prod_{j\in J}\lt \bE\left(\ze^{b_{i,j}z_j}\right)\rt
=\sum_{i\in I}\prod_{j\in J}\lt1-\sum_{l=1}^{k_j} s_{l,j}(1-\ze^{b_{i,j}t_{l,j}})\rt
$$
where for each $j \in J$, $1 \le k_j \le m-1$ and $z_j \in \mr{Ran}(R,\al)$ satisfies
$$
\bP(z_j=0)=1-\sum_{l=1}^{k_j}s_{l,j} \in (\al, 1-\al] \text{ and } \bP(z_j=t_{l,j})=s_{l,j} \in [\al/(m-1)!,\al]
$$
for distinct elements $t_{1,j}, \ldots, t_{k_j,j} \in R \setminus \{0 \}$ that generate $R$.
\end{lem}

\begin{proof}
As in the proof of Lemma \ref{lem: extreme point}, the absolute value of $\sum_{i\in I}\prod_{j\in J}\bE\left(\ze^{b_{i,j}x_j}\right)$ attains its maximum at a point $(x_j)_{j \in J}$ where each $x_j$ corresponds to the extreme point of the set 
$$
C_{m,\al}:=\left\{\mb{s} = (s_1,\cdots,s_{m})\in[0,1]^m:\sum_{i=1}^{m} s_i=1 \text{ and }\sum_{i \equiv a \pmod{p}} s_i \le 1-\al \text{ for all } p \in \mc{P} \text{ and } 0 \le a \le p-1 \right\}.
$$
Let $\mb{s} = (s_1,\cdots,s_{m})$ be an extreme point of $C_{m,\al}$. Assume that $s_i, s_j > \al$ for some $i \ne j$. Define $\mb{s'} = (s'_1,\cdots,s'_{m})$ and $\mb{s''} = (s''_1,\cdots,s''_{m})$ in $C_{m,\al}$ by
$$
s'_t = \left\{\begin{matrix}
s_t & t \ne i,j, \\
s_i+s_j-\al & t=i,\\
\al & t=j
\end{matrix}\right. \quad \text{and} \quad s''_t = \left\{\begin{matrix}
s_t & t \ne i,j, \\
\al & t=i,\\
s_i+s_j-\al & t=j.
\end{matrix}\right. 
$$
Then $\mb{s} = \eta\mb{s'}+(1-\eta)\mb{s''}$ for some $\eta \in (0,1)$, contradicting the extremality of $\mb{s}$. Since $\sum_{i=1}^{m} s_i = 1 > m \al$, it follows that for every extreme point $\mb{s} = (s_1,\cdots,s_{m})$ of $C_{m,\al}$ there exists a unique $i \in [m]$ such that $s_i > \al$. 

Now fix $j\in J$ and consider the corresponding element $x_j\in \mr{Ran}(R,\al)$. Then there exists $i_j \in [m]$ such that $\bP(x_j = i_j) > \al$ and $\bP(x_j = x) \le \al$ for all $x \ne i_j$. Replacing $x_j$ by $x_j-i_j$ (which leaves the sum $\sum_{i\in I}\prod_{j\in J}\lt\bE\left(\ze^{b_{i,j}z_j}\right)\rt$ unchanged), we may assume that $i_j=0$. This yields the desired element $z_j$.
Suppose that $\langle t_{1,j},\cdots,t_{k_j,j}\rangle= m_0R \ne R$. Then $m_0R\subseteq pR$ for some $p\in\mc{P}$ so
$$
\bP(z_j\equiv0\bmod{p})=\bP(z_j=0)+\sum_{l=1}^{k_j}s_{l,j}=1,
$$
which is impossible. 

It remains to prove that $s_{l,j}\ge\frac{\al}{(m-1)!}$. For each $p\in\mc{P}$, define 
$$
C_{p,j}:=\{l \in [k_j] : t_{l,j}\not\equiv0\bmod{p}\}.
$$
Then
$$
\sum_{l\in C_{p,j}}s_{l,j}=1-\bP(z_j\equiv0\bmod{p})\ge\al.
$$
Let $\mc{P}_j\subset\mc{P}$ be a subset of primes in $\mc{P}$ for which the inequality becomes an equality. The fact that $\sum_{l\in C_{p,j}}s_{l,j}=\al$ for every $p \in \mc{P}_j$ can be written in matrix form as
$$
M_j \begin{pmatrix}
s_{1,j}\\\vdots\\s_{k_j,j}
\end{pmatrix}=\begin{pmatrix}
\al\\\vdots\\\al
\end{pmatrix},
$$
where $M_j$ is a (not necessarily square) matrix with entries in $\{ 0,1 \}$. Since $z_j$ is an extreme point, the homogeneous system $M_j v=0$ admits only the trivial solution; hence $M_j$ has full rank $k_j$. In other words, there exists a $k_j\times k_j$ minor $M_{0,j}$ of $M_j$ such that
$$
\begin{pmatrix}
s_{1,j}\\\vdots\\s_{k_j,j}
\end{pmatrix}=M_{0,j}^{-1}\begin{pmatrix}
\al\\\vdots\\\al
\end{pmatrix}.
$$
Since $s_{l,j} > 0$, $M_{0,j}^{-1} \in \M_{k_j}\left(\frac{1}{\det M_{0,j}}\bZ\right)$ and $\lt\det M_{0,j} \rt \le k_j!\le (m-1)!$, the above identity implies that $s_{l,j} \ge \frac{\al}{(m-1)!}$.
\end{proof}

With this preparation, the remaining estimates follow by the same argument as in Section \ref{Sec4}, once the appropriate analogue of Lemma \ref{lem_ine1} is established. We now prove this analogue.

\begin{lem}\label{lemma_ine2}
Let $1 \le k \le m-1$, $s_1,\cdots,s_k\in[\al_n/(m-1)!,\al_n]$, and assume that $t_1,\cdots,t_k \in R \setminus \{0\}$ generate $R$. Then for every nonzero $r\in R$ and all sufficiently large $n$, 
$$
\lt1-\sum_{i=1}^ks_i(1-\ze^{t_ir})\rt
\le1-2\left(1-\frac{kc\log n}{n}\right) \sum_{i=1}^ks_i\sin^2\left(\frac{\pi t_ir}{m}\right)
\le1-\frac{8c_1}{m^2\cdot (m-1)!}\frac{\log n}{n}.
$$
\end{lem}

\begin{proof}
Write $S:=\sum_{i=1}^k s_i$ and $\theta_i:=\frac{\pi t_i}{m}$, so that $\ze^{t_ir}=\cos(2\theta_ir) + i\sin(2\theta_ir)$. Suppose that $n$ is sufficiently large so that $S<1$, then
\begin{align*}
\lt1-\sum_{i=1}^ks_i(1-\ze^{t_ir})\rt
&=\lt (1-S+\sum_{i=1}^{k} s_i\cos(2\theta_ir)) + i\sum_{i=1}^{k} s_i\sin(2\theta_ir) \rt\\
&=\sqrt{(1-S)^2+\sum_{i=1}^ks_i^2+2(1-S)\sum_{i=1}^k s_i\cos(2\theta_ir)+2\sum_{1\le i<j\le k}s_is_j\cos\left(2(\theta_i-\theta_j)r\right)}\\
&\le\sqrt{(1-S)^2+\sum_{i=1}^ks_i^2+2(1-S)\sum_{i=1}^ks_i\cos(2\theta_ir)+2\sum_{1\le i<j\le k}s_is_j}\\
&=\sqrt{1-4(1-S)\sum_{i=1}^ks_i\sin^2(\theta_ir)}\\
&\le1-2(1-S)\sum_{i=1}^ks_i\sin^2(\theta_ir).
\end{align*}
The upper bound $S \le k\al_n$ proves the first inequality. 

Since $t_1, \ldots, t_k$ generate $R$, for every nonzero $r \in R$ there exists $i\in[k]$ where $\sin(\theta_ir)\ne 0$. Using the inequalities $s_i \ge \al_n/(m-1)!$ and $\sin^2(\theta_i r)\ge (2/m)^2$, we obtain
\begin{align*}
\lt1-\sum_{i=1}^{k}s_i(1-\ze^{t_ir})\rt
&\le1-2\left(1-\frac{kc\log n}{n}\right)\frac{\al_n}{(m-1)!} \sum_{i=1}^{k}\sin^2\left(\theta_ir\right)\\
&\le1-\frac{2c_1}{(m-1)!}\frac{\log n}{n}\cdot\left(\frac{2}{m}\right)^2\\
&=1-\frac{8c_1}{m^2\cdot (m-1)!}\frac{\log n}{n}
\end{align*}
for all sufficiently large $n$, which completes the proof.
\end{proof}

We may apply Lemma \ref{lemma_ine2} to the cases $\mc{C}_1$, $\mc{C}_{3,a}^W$, $\mc{C}_{3,b}^W$ and $\mc{C}_{3,d}^W$. The proof in the case $\mc{C}_2$ goes through verbatim.
In the remaining case $\mc{C}_{3,c}^W$, we redefine $S_y$ by
$$
S_y:=\sum_{x\in X}\frac{1}{\al_n}\left(\sum_{l=1}^{k_{x,y}}s_{l,x,y}\sin^2\frac{\pi t_{l,x,y,n}B(\mb{g}_x,\mb{g}_y)}{m}+\sum_{l=1}^{k_{y,x}}s_{l,y,x}\sin^2\frac{\pi t_{l,y,x,n}B(\mb{g}_y,\mb{g}_x)}{m}\right)\le2m\lt X\rt.
$$
Here we write $W=\mb{g}+H$ as before, and for each ordered pair $(x,y)$ let $k_{x,y}$, $s_{l,x,y}$ and $t_{l,x,y,n}$ be the parameters given by Lemma \ref{lem: extreme point2}.

Fix $\mathbf g_x\notin W$. By condition \ref{bilinear_cond2} in Section \ref{Sub23}, we may assume without loss of generality that the map  $B(\mb{g}_x,\mb{g}+*)$ is a non-constant affine map on $H$ (and also $B(\mb{g}+*,\mb{g}_x)$ if $\kappa=2$). Write $$\{B(\mb{g}_x,w):w\in W\} = a_0+m_0R$$ 
with $m_0<m$ a divisor of $m$ and $a_0 \in R$. 
Then
\begin{align*}
\sum_{\mb{g}_y\in W}S_y&\ge\kappa\sum_{x\in X}\sum_{l=1}^{k_{x,y}}\frac{s_{l,x,y}}{\al_n}\cdot\sum_{\mb{g}_y\in W}\sin^2\frac{\pi t_{l,x,y,n}B(\mb{g}_x,\mb{g}_y)}{m}
\\&\ge\frac{\kappa\lt G\rt}{2}\sum_{x\in X} \sum_{\substack{1 \le l \le k_{x,y} \\m_0t_{l,x,y,n}\ne0}}\frac{s_{l,x,y}}{\al_n}
\\&\ge\frac{\kappa\lt G\rt}{2}\sum_{x\in X}\frac{1}{\al_n}\sum_{\substack{1 \le l \le k_{x,y} \\ p \,\nmid\, t_{l,x,y,n}}}s_{l,x,y}
\\&\ge\frac{\kappa\lt X\rt\lt G\rt}{2},
\end{align*} 
where $p \in \mc{P}$ is a prime divisor of $m/m_0>1$. By following the proof of Proposition \ref{prop_case3c} with this modification, we see that the argument also works in the present case, and this completes the proof of Theorem \ref{moment thm finite prime case}. 

\bigskip
\section*{Acknowledgments}

Jungin Lee was supported by the National Research Foundation of Korea (NRF) grant funded by the Korea government (MSIT) (No. RS-2024-00334558 and No. RS-2025-02262988). Myungjun Yu was supported by the National Research Foundation of Korea (NRF) grant funded by the Korea government (MSIT) (RS-2025-23525445).



\begin{thebibliography}{99}

\bibitem{Bol01}
B. Bollobás, Random Graphs, Cambridge University Press, Cambridge, 2001.

\bibitem{CY23}
G. Cheong and M. Yu, The distribution of the cokernel of a polynomial evaluated at a random integral matrix, arXiv:2303.09125, to appear in Amer. J. Math.



\bibitem{CT06}
T. M. Cover and J. A. Thomas, Elements of Information Theory, Second Edition, John Wiley \& Sons, 2006.


\bibitem{FG15}
J. Fulman and L. Goldstein, Stein’s method and the rank distribution of random matrices over finite fields, Ann. Probab. 43 (2015), no. 3, 1274--1314.


\bibitem{Lee23}
J. Lee, Universality of the cokernels of random $p$-adic Hermitian matrices, Trans. Amer. Math. Soc. 376 (2023), no. 12, 8699--8732.


\bibitem{Lee25}
J. Lee, Sharp threshold for universality of cokernels of random matrices over finite fields, arXiv:2511.13070.

\bibitem{NVP24}
H. H. Nguyen and R. Van Peski, Universality for cokernels of random matrix products, Adv. Math. 438 (2024), 109451.

\bibitem{NW22}
H. H. Nguyen and M. M. Wood, Random integral matrices: universality of surjectivity and the cokernel, Invent. Math. 228 (2022), 1--76.

\bibitem{NW25}
H. H. Nguyen and M. M. Wood, Local and global universality of random matrix cokernels, Math. Ann. 391 (2025), no. 4, 5117--5210.

\bibitem{Woo17}
M. M. Wood, The distribution of sandpile groups of random graphs, J. Amer. Math. Soc. 30 (2017), no. 4, 915--958.

\bibitem{Woo19}
M. M. Wood, Random integral matrices and the Cohen--Lenstra heuristics, Amer. J. Math. 141 (2019), no. 2, 383--398.

\bibitem{Woo22}
M. M. Wood, Probability theory for random groups arising in number theory, Proc. Int. Cong. Math. 2022, Vol. 6, 4476--4508.
\end{thebibliography}
\end{document}